\documentclass[11 pt]{article}
\usepackage[english]{babel}
\usepackage[a4paper,left=2cm,right=2cm,top=2.5cm,bottom=2.5cm]{geometry}
\usepackage{graphicx}
\usepackage[pdfborder={0 0 0}]{hyperref}
\usepackage[utf8]{inputenc}

\usepackage{geometry}
\usepackage{mathtools}
\usepackage{amssymb}
\usepackage{dsfont}
\usepackage{amsmath}
\usepackage{url}
\usepackage{float}
\usepackage{comment}
\usepackage[T1]{fontenc}
\usepackage{bm}
\usepackage{graphicx}
\usepackage[normalem]{ulem}
\usepackage{commath}
\usepackage{amsthm}
\usepackage{mathrsfs}  
\usepackage{amsmath,amscd}
\usepackage{hyperref}
\usepackage{bm}
\usepackage{array}
\usepackage[nottoc,notlot,notlof]{tocbibind}
\usepackage[title]{appendix}
\usepackage{tikz}
\usepackage[title]{appendix}
\usepackage{xcolor}
\usepackage{leftidx}
\usepackage{ stmaryrd }
\usetikzlibrary{fadings}

\usepackage{pgfplots}

\usepackage{cleveref}

\usepackage{tikz-cd}
\usetikzlibrary{matrix,arrows}
\usepackage{subfiles}

\renewcommand{\subset}{\subseteq}

\newcommand{\Tr}{\mathrm{Tr}}

\newcommand{\End}{\mathrm{End}}

\newcommand{\rk}{\mathrm{rk}}

\newcommand{\inv}{{}^{-1}}

\newcommand{\coker}{\mathrm{coker}}

\newcommand{\mc}{\mathcal}
\newcommand{\mf}{\mathfrak}
\newcommand{\mb}{\mathbb}
\newcommand{\linfty}{L_\infty}

\newcommand{\Addresses}{{
\bigskip
\footnotesize

\textsc{KU Leuven, Department of Mathematics, Celestijnenlaan 200B box 2400, 3001 Leuven, Belgium} \par \nopagebreak
\textit{E-mail address}:\href{mailto:karandeep.singh@kuleuven.be}{karandeep.singh@kuleuven.be}
}}

\makeatletter
\newcommand{\opnorm}{\@ifstar\@opnorms\@opnorm}

\makeatother

\pgfplotsset{compat=1.10}
\usepgfplotslibrary{fillbetween}
\usetikzlibrary{patterns}

\newcommand{\Hom}{\mathrm{Hom}}

\theoremstyle{plain}
\newtheorem{thm}{Theorem}[section]
\newtheorem*{thm*}{Theorem}
\newtheorem{lem}[thm]{Lemma}
\newtheorem{prop}[thm]{Proposition}
\newtheorem{cor}[thm]{Corollary}

\theoremstyle{definition}
\newtheorem{defn}[thm]{Definition} 
\newtheorem{exmp}[thm]{Example}
\newtheorem{rmk}[thm]{Remark}

\newtheorem*{conv*}{Convention}
\newtheorem*{prob*}{Problem}

\title{On the universal $L_\infty$-algebroid of linear foliations}
\author{Karandeep Singh}
\date{}
\begin{document}
\maketitle

\abstract{We compute an $L_\infty$-algebroid structure on a projective resolution of some classes of singular foliations on a vector space $V$ induced by the linear action of some Lie subalgebra of $\mf {gl}(V)$. This $L_\infty$-algebroid provides invariants of the singular foliations, and also provides a constant-rank replacement of the singular foliation. We do this by first explicitly constructing projective resolutions of the singular foliations induced by the natural linear actions of endomorphisms of $V$ preserving a subspace $W\subset V$, the Lie algebra of traceless endomorphisms, and the symplectic Lie algebra of endomorphisms of $V$ preserving a non-degenerate skew-symmetric bilinear form $\omega$, and then computing the $L_\infty$-algebroid structure. We then generalize these constructions to a vector bundle $E$, where the role of the origin is now taken by the zero section $L$.\\

We then show that the fibers over a singular point of a projective resolution of any singular foliation can be computed directly from the foliation, without needing the projective resolution. For linear foliations, we also provide a way to compute the action of the isotropy Lie algebra in the origin on these fibers, without needing the projective resolution.}
\tableofcontents

\section{Introduction}\label{a}
Let $M$ be a smooth manifold, equipped with a singular foliation $\mc F$. By singular foliation, we mean a subsheaf $\mc F$ of the sheaf of vector fields $\mf X$ on $M$ such that
\begin{itemize}
    \item[a)] for all $U\subset M$ open, $\mc F(U)$ is a $C^\infty(U)$-submodule of $\mf X(U)$,
    \item[b)] for all $U\subset M$ open and $X,Y \in \mc F(U)$ we have $[X,Y] \in \mc F(U)$,
    \item[c)] for all $x \in M$, there exists an open subset $U_x$ of $M$ containing $x$, such that $\mc F(U_x)$ is a finitely generated $C^\infty(U_x)$-module.
\end{itemize}
This definition of singular foliations was used in \cite{lavau2016,univlinfty}. An equivalent definition, using compactly supported vector fields, appeared in \cite{sussmann,holgpd} among other places. This equivalence was shown in \cite[Proposition 2.1.9]{wang}, and the construction of the sheaf out of compactly supported vector fields appeared in \cite{sheafcomp}. \\
In \cite{univlinfty}, it was shown that under certain conditions on $\mc F$ one can associate an $L_\infty$-algebroid over $M$ to $(M,\mc F)$. Here an $L_\infty$-algebroid is a non-positively graded vector bundle $E =\bigoplus_{i \in \mb Z_{\leq 0}} E_i $, with a collection of multibrackets $\{\ell_k:\Gamma(\wedge^k E)\to \Gamma(E)\}_{k\geq 1}$ where $\ell_k$ has degree $2-k$, and a vector bundle map $\rho:E_0 \to TM$ intertwining $\ell_2$ with the Lie bracket of vector fields called the anchor, satisfying some quadratic identities. $L_\infty$-algebroids were first defined in \cite{voronov} as higher analogues of Lie algebroids. When $M = \{\ast\}$ is a single point, the definition reduces to that of a non-positively graded $L_\infty$-algebra, which appeared in \cite{shlie} as strongly homotopy Lie algebra.  For the definition and important properties, we refer to \cite[Section 2.1]{modclasslavau}. 

The construction of \cite{univlinfty} can be broken into two parts:
\begin{itemize}
    \item [i)] Choosing a resolution of $\mc F$ in the category of $C^\infty_M$-modules by finitely generated projective modules\footnote{As we work in the smooth category, this is equivalent to choosing a resolution of $\mc F(M)$ in the category of $C^\infty(M)$-modules by sections of vector bundles. We therefore do not distinguish between the sheaf and its global sections.},
    \item [ii)] Constructing an $L_\infty$-algebroid structure on the complex given by the resolution.
\end{itemize}

In step i) the conditions posed on $\mc F$ are used. Neither of the steps is constructive, but plenty of examples are given. Because of i), the $L_\infty$-algebroid constructed in ii) satisfies a universality property, which implies uniqueness up to a notion of homotopy (\cite[Corollary 2.9]{univlinfty}). It will therefore be referred to as a \emph{universal} $L_\infty$-algebroid of $\mc F$. Because of the uniqueness up to homotopy, this $L_\infty$-algebroid captures invariants of the singular foliation $\mc F$. Moreover, it allows to replace the singular foliation by a collection of constant-rank objects, which provides a framework to extend some results from the theory of Lie algebroids to singular foliations. Further, knowing a universal $L_\infty$-algebroid of a singular foliation allows to compute the modular class of a singular foliation as in \cite{modclasslavau}.\\

In this article, we generalize the example of vector fields vanishing in the origin given in \cite{univlinfty}. This foliation is induced by the canonical linear $\mf {gl}(V)$-action on $V$. The universal $L_\infty$-algebroid given in \cite[Example 3.99]{univlinfty} only has two nonzero operations, turning it into a differential graded Lie algebroid (dg-Lie algebroid): $\ell_k = 0$ for $k\geq 3$.
This raises several questions:
\begin{itemize}
    \item[1)] Can we construct the universal $L_\infty$-algebroids for linear actions of other Lie algebras explicitly?
    \item[2)] Can this approach be generalized to higher-dimensional leaves, with the corresponding isotropy Lie algebra?
    \item[3)] Does the universal $L_\infty$-algebroid for such a foliation always admit a dg-Lie algebroid structure (i.e. an $L_\infty$-algebroid structure for which only the unary and binary brackets are non-zero)?
\end{itemize}
\subsection*{Main results}
We address these questions in the following examples: 
\begin{itemize}
\item[-] The Lie subalgebra $\mf {gl}(V,W)\subset \mf {gl}(V)$ for a given subspace $W\subset V$,
\item[-] the Lie subalgebra $\mf {sl}(V)\subset \mf {gl}(V)$ of traceless endomorphisms
\item[-] the Lie subalgebra $\mf {sp}(V,\omega)\subset \mf {gl}(V)$ of endomorphisms preserving a non-degenerate skew-symmetric 2-form $\omega \in \wedge^2 V^\ast$.
\end{itemize}
All three questions above have a positive answer in the cases $\mf {gl}(V,W)$ and $\mf {sl}(V)$. We answer questions 1) and 2) partially in the case of $\mf {sp}(V,\omega)$, and we do not know the answer to question 3) in this case.\\
The resolutions of the module $\mc F$ we construct are \emph{minimal} at the origin, which means that all differentials, being vector bundle maps, vanish at the origin. An advantage of this is that two $L_\infty$-algebroid structures constructed on minimal resolutions are not only homotopy equivalent, but actually $L_\infty$-isomorphic in a neighborhood of the origin, as explained at the end of section \ref{sec:vf00}. \\
In section \ref{sec:0dimleaf} we address questions 1) and 3).
\begin{itemize}
    \item[-] In section \ref{sec:vf00} we recall the construction for $\mf {gl}(V)$, as given in \cite{univlinfty}.
    \item[-] In section \ref{sec:subsp} we consider the case of $\mf {gl}(V,W)$, which induces the foliation generated by the linear vector fields tangent to the subspace $W$. We give a geometric resolution and describe an $L_\infty$-algebroid structure with only a unary and binary bracket in Proposition \ref{prop:subspres} yielding a positive answer to question 3).
    \item[-] In section \ref{sec:vol0} we consider the case of $\mf {sl}(V)$, which induces the foliation generated by linear vector fields preserving a constant volume form on $V$. We compute a geometric resolution in Proposition \ref{prop:ressl}, and describe an $L_\infty$-algebroid structure with only a unary and binary bracket in Proposition \ref{prop:slbracket} yielding a positive answer to question 3).
    \item[-] In section \ref{sec:linsymp} we fix a non-degenerate element $\omega\in \wedge^2 V^\ast$ and consider the case of $\mf {sp}(V,\omega)$. We compute the geometric resolution in Proposition \ref{prop:spres}, and give a binary bracket in Proposition \ref{prop:spbracket} depending on a map $r^\omega$ we chose. We show that this bracket does not satisfy the Jacobi identity, and give an expression for the ternary bracket. In the appendix \ref{appendix1} we investigate if the binary brackets can be simplified by picking $r^\omega$ to be a cochain map in some degrees, and show that this cannot be done when $V$ is 4-dimensional. The answer to question 3) remains inconclusive in this case.
\end{itemize}
In section \ref{sec:hdleaf} we the higher-dimensional analogues of the abovementioned cases and address the corresponding questions 2) and 3). In each of the cases the results of the earlier sections generalize.
\begin{itemize}
    \item[-] In section \ref{sec:zerosec} we consider the foliation of vector fields on a vector bundle $E$ which are tangent to the zero section. We compute the geometric resolution in Proposition \ref{prop:resgle}, and describe an $L_\infty$-algebroid structure in Proposition \ref{prop:glebracket}.
    \item[-] In section \ref{sec:subbundle} we consider the foliation of vector fields on a vector bundle which are tangent to a vector subbundle, of which the zero section is a special case. The geometric resolution and $L_\infty$-algebroid structure are given in Proposition \ref{prop:subbpres}.
    \item[-] In section \ref{sec:vol1} we consider the foliation the foliation on an orientable vector bundle $E \to L$, with non-vanishing section $\mu\in \Gamma(\wedge^n E)$, where $n = \rk (E)$, generated by the linear vector fields which preserve $\mu$. We give the geometric resolution in proposition \ref{prop:sleres}, and the $L_\infty$-algebroid structure in \ref{prop:slebracket}.
    \item[-] In section \ref{sec:sympleaf} we consider the foliation on a vector bundle $E\to L$ generated by the linear vector fields which preserve a non-degenerate $\omega\in \Gamma(\wedge^2 E)$. The projective resolution is given in \ref{prop:resspe}, and a binary bracket of the $L_\infty$-algebroid structure in Proposition \ref{prop:spebracket}.
\end{itemize}
Finally, in section \ref{sec:generalfol} we consider a general foliation $\mc F$ on a vector space $V$, for which the origin $p$ is a singular point. We show that the fibers over $p$ of any geometric resolution which is minimal at the origin can be computed directly from $\mc F$, \emph{without needing to find a geometric resolution} (Proposition \ref{prop:toriso}). In the case that $\mc F$ is linear, we additionally show that part of the structure of the isotropy $L_\infty$-algebra (see \cite[Section 4.2]{univlinfty}), which is an invariant of the foliation $\mc F$, can be recovered from the foliation directly (Proposition \ref{prop:repisoliealg}).
\paragraph{\bf Acknowledgements} 
We thank Marco Zambon for fruitful discussions and helpful comments. We thank Wouter Castryck and Robin van der Veer for helpful discussions. We thank Sylvain Lavau for providing useful comments. We acknowledge the FWO and FNRS under EOS projects G0H4518N and G0I2222N. 
\section{Zero-dimensional leaves}\label{sec:0dimleaf}
In this section, we compute a universal $L_\infty$-algebroid for some classes of singular foliations generated by some Lie subalgebra of the Lie algebra of linear vector fields on a vector space $V$, addressing questions 1) and 3) from Section \ref{a}.\\

In section \ref{sec:vf00}, we recall Example 3.99 of \cite{univlinfty}, dealing with the foliation of vector fields vanishing at the origin $p\in V$, which is induced by the standard $\mf{gl}(V)$-action.\\

In section \ref{sec:subsp}, we pick a subspace $W\subset V$, and consider the foliation generated by all linear vector fields tangent to $W$, which is the foliation induced by the action of endomorphisms of $V$ which preserve $W$.\\

In section \ref{sec:vol0}, we consider the foliation generated by linear vector fields preserving any constant volume form on $V$, yielding the foliation generated by the standard $\mf{sl}(V)$-action.\\

In section \ref{sec:linsymp}, we consider a non-degenerate element $\omega \in \wedge^2 V^\ast$, and consider the foliation generated by linear vector fields preserving $\omega$, which is the foliation induced by the standard $\mf{sp}(V,\omega)$-action.
\begin{conv*}
Throughout this section, for $W$ a finite-dimensional real vector space, we will consider trivial vector bundles $W\times V$ over a finite-dimensional real vector space $V$. Its global sections will be denoted by $\Gamma(W)$.\\
Unless stated otherwise, repeated indices will be summed over.
\end{conv*}
\subsection{Vector fields vanishing at the origin}\label{sec:vf00}
In this section we recall Example 3.99 of \cite{univlinfty}. Let $V$ be a real vector space of dimension $n\geq 0$, and let 
\begin{equation}\label{eq:0inorigin}
\mc F_0(V) = \{ X \in \mf X(V) \mid X(0) = 0\}
\end{equation}
be the submodule of vector fields on $V$ given by the vector fields vanishing in the origin. It is easy to see that it is a singular foliation.
A resolution of $\mc F_0$ can be constructed using the following lemma.
\begin{lem}\label{lem:koszulexact}
The complexes
\begin{equation} 
\begin{tikzcd}\label{eq:koszexact1}
0\arrow{r}& \Gamma(\wedge^n V^\ast) \arrow{r}{d_{n}} & \Gamma(\wedge^{n-1}V^\ast) \arrow{r}{d_{n-1}} & \dots \arrow{r}{d_{2}} & \Gamma(V^\ast) \arrow{r}{\rho } & I_q \arrow{r}& 0
\end{tikzcd},
\end{equation}
\begin{equation}
\begin{tikzcd}\label{eq:koszexact2}
0\arrow{r}& \Gamma(\wedge^n V^\ast) \arrow{r}{d_{n}} & \Gamma(\wedge^{n-1}V^\ast) \arrow{r}{d_{n-1}} & \dots \arrow{r}{d_{2}} & \Gamma(V^\ast) \arrow{r}{d_1} & C^\infty(V) \arrow{r}{\text{ev}_q}& \mb R\arrow{r}& 0
\end{tikzcd}
\end{equation}
are exact. Here for $k = 1,\dots, n$, $d_k:\Gamma(\wedge^k V^\ast) \to \Gamma(\wedge^{k-1} V^\ast)$ and $\rho:\Gamma(V^\ast) \to I_q$ are the contraction with the Euler vector field $x^i\partial_{x^i}$, $I_q$ is the ideal of functions vanishing at the origin $q\in V$ and $\text{ev}_q$ is the evaluation of a function at $q = 0$.
In particular, the complexes remain exact when applying the functor $-\otimes_{C^\infty(V)} \Gamma(W)$ for some vector bundle $W\times V\to V$.
\end{lem}
Taking $W = V$ and tensoring $\eqref{eq:koszexact1}$ with $\Gamma(V)\cong \mf X(V)$ we obtain the exact sequence
\begin{equation}\label{eq:resgln}
\begin{tikzcd}
0\arrow{r}& \Gamma(\wedge^n V^\ast\otimes V) \arrow{r}{d_{n}} & \Gamma(\wedge^{n-1}V^\ast\otimes V) \arrow{r}{d_{n-1}} & \dots \arrow{r}{d_{2}} & \Gamma(V^\ast\otimes V) \arrow{r}{\rho} & I_p\mf X(V) = \mc F_0(V) \arrow{r}& 0.
\end{tikzcd}
\end{equation}
Here, and in the rest of this article we use the convention that $\Gamma(V^\ast \otimes V)$ sits in degree $0$, and the differential $d_\bullet$ has degree $1$.

An $L_\infty$-algebroid structure on $\Gamma(\wedge^\bullet V^\ast \otimes V)$ can be given as follows: for the unary bracket, we take $d_\bullet$, as in \eqref{eq:resgln}. For the binary bracket we take the Nijenhuis-Richardson bracket: For $1\leq k_1,k_2 \leq n$ define
\begin{align}
[-,-]: \Gamma\left(\wedge^{k_1} V^\ast \otimes V \right)\times \Gamma\left( \wedge^{k_2} V^\ast \otimes V\right) &\to\Gamma\left( \wedge^{k_1+k_2-1}V^\ast \otimes V\right)\nonumber \\
(f_1\cdot (\phi_1\otimes w_1), f_2\cdot(\phi_2\otimes w_2)) \mapsto& f_1f_2\cdot (\phi_1 \iota_{w_1}(\phi_2) \otimes w_2 - (-1)^{(k_1-1)(k_2-1)} (1 \leftrightarrow 2))\nonumber \\
&+\left(\delta_{k_1,1}f_1\rho(\phi_1\otimes w_1)(f_2)\cdot (\phi_2 \otimes w_2) - (1 \leftrightarrow 2)\right).\label{eq:glnbracket}
\end{align}
One can check that this defines a dg-Lie algebroid over $V$ for which the image of $\rho$ is exactly $\mc F_0$. We denote it by $L_\infty(\mc F_0)$.

Note that the differentials $d_p$ vanish at the origin for $p = 2,\dots, n$. This implies that any $L_\infty$-algebroid structure with the same property is $L_\infty$-isomorphic to the one above in a neighborhood of the origin: by \cite[Corollary 2.9]{univlinfty}, any two $L_\infty$-algebroid structures are homotopy equivalent by an $L_\infty$-morphism $\Phi$. By minimality and \cite[Lemma 4.13iii)]{univlinfty}, this implies that the homotopy equivalence is an isomorphism in the origin. As invertibility is an open condition it follows that it is an isomorphism in a neighborhood of the origin. 
\subsection{Linear vector fields preserving a subspace}\label{sec:subsp}
Let $V$ be a real vector space of dimension $n$, and $W \subset V$ a linear subspace. Let $$
\
\mc F_W(V) :=\{X \in \mc F_0(V) \mid X(I_{W}) \subset I_{W} \}
$$
be the $C^\infty(V)$-submodule of linear vector fields tangent to the subspace $W$. This is a singular foliation, and is induced by the action of the Lie subalgebra $\mf {gl}(V,W)$ of $\mf{gl}(V)$ given by
$$
\mf {gl}(V,W) = \{A \in \mf {gl}(V) \mid A(W) \subset W\},
$$
the endomorphisms of $V$ preserving $W$. The leaves of this foliation consist of the origin, the connected components of $W\setminus\{0\}$, and the connected components of $V\setminus W$. \\
\begin{exmp}
Let $V = \mb R^2, W = \{(x,0)\in \mb R^2\mid x \in \mb R\}$. Then $\mc F_W(V)$ is generated by the vector fields $ x \partial_x, y\partial_x, y\partial_y$, and the leaves are the positive $x$-axis, the origin, the negative $x$-axis, the upper half plane and the lower half plane. In this case $\mf {gl}(V,W)$ consists of all upper triangular matrices.
\end{exmp}
We can describe a minimal universal $L_\infty$-algebroid of $\mc F_W$ as a $L_\infty$-subalgebroid of $L_\infty(\mc F_0)$. In particular, it will again be a dg-Lie algebroid.
\begin{defn}
Let $j \in \{1,\dots, n\}$. Define $K_j \subset \wedge^j V^\ast \otimes V = \Hom(\wedge^j V,V)$ by
$$
K_j := \{\phi \in \wedge^j V^\ast \otimes V\mid \forall w\in W,\forall v_1,\dots, v_{j-1} \in V : \phi(w,v_1,\dots, v_{j-1}) \in W\}.
$$
\end{defn}
\begin{prop}\label{prop:subspres}
\begin{itemize}
    \item []
    \item[i)] The differential $$d_j:\Gamma(\wedge^j V^\ast \otimes V) \to \Gamma(\wedge^{j-1}V^\ast \otimes V) $$ as in \eqref{eq:resgln} restricts to a map
    $$
    d_j:\Gamma(K_j)\to \Gamma(K_{j-1}).
    $$
    \item[ii)] The bracket \eqref{eq:glnbracket} restricts to the subspaces $\Gamma(K_j)$.
    \item[iii)] The subcomplex 
    \begin{equation}\label{eq:resglvw}
        \begin{tikzcd}
        0 \arrow{r}& \Gamma(K_n) \arrow{r}{d_n} &\Gamma(K_{n-1}) \arrow{r}{d_{n-1}}&\dots \arrow{r}{d_2}& \Gamma(K_1)\arrow{r}{\rho_W} & \mc F_W(V) \arrow{r}& 0
        \end{tikzcd}
    \end{equation}
    is exact, where $\rho_W = \left. \rho \right|_{\Gamma(K_1)}$
\end{itemize}
Consequently, $\Gamma(K_\bullet)$ with the restrictions of $d_\bullet$ and $[-,-]$ is a minimal universal $L_\infty$-algebroid of the foliation $\mc F_W$.
\end{prop}
\begin{proof}
Items i) and ii) are straightforward computations. For item iii), fix a complement $C$ of $W$ in $V$. Then $K_i$ can be identified with $\wedge^i V^\ast \otimes W \oplus \wedge^i C^\ast \otimes C$, and the complex \eqref{eq:resglvw} decomposes as
$$
(\Gamma(K_\bullet),\partial) = (\Gamma(\wedge^\bullet V^\ast \otimes W) \oplus \Gamma(\wedge^\bullet C^\ast \otimes C), \partial_W + \partial_C),
$$
where 
$$
\partial_W(\phi) = x^i\iota_{e_i}(\phi), 
$$
for $\phi \in \Gamma(\wedge^i V^\ast\otimes W)$, and $\{e_i\}_{i=1}^n$ is a basis for $V$, with linear coordinates $\{x^i\}_{i=1}^n$. 
For $\psi\in \Gamma(\wedge^i C^\ast \otimes C)$, 
$$
\partial_C(\psi) = y^i\iota_{f_i}(\psi),
$$
where $\{f_i\}_{i=1}^r$ is a basis for $C$, and $\{y^i\}_{i=1}^r$ are the corresponding linear coordinates. By Lemma \ref{lem:koszulexact}, both are exact, concluding the proof.
\end{proof}
\subsection{Vector fields preserving a volume form}\label{sec:vol0}
The next choice for a Lie algebra $\mf g$ acting linearly on a vector space $V$ we consider is $\mf g = \mf {sl}(V)$, the Lie algebra of traceless endomorphisms. Observe that the partition of $V$ is identical to the case of $\mf {gl}(V)$, but that the underlying submodules of $\mf X_V$ are different. Let $\mu \in \wedge^n V^\ast$ be a non-zero element, and denote the foliation given by the action of $\mf{sl}(V)$ by $\mc F_\mu$. 
\subsubsection{The projective resolution}
As it is in general not possible to restrict a projective resolution of a module to a submodule, one cannot directly get a projective resolution of the module of vector fields generated by the action of $\mf {sl}(V)$, by restricting all modules to live over $\mf {sl}(V)$. But for the most part, the resolution we will construct is related to the one given in \eqref{eq:resgln}. Consider the following diagram of $C^\infty(V)$-modules:
\begin{equation}\label{diag:parttr}
\begin{tikzcd}
\Gamma(\wedge^2 V^\ast \otimes V) \arrow{r}{d_2}& \Gamma(V^\ast \otimes V)\arrow{d}{\Tr}\\
\Gamma(V^\ast) \arrow{r}{\partial_1} & \Gamma(\mb R)
\end{tikzcd},
\end{equation}
where $\partial_1:\Gamma(V^\ast) \to \Gamma(\mb R)$ is the contraction with the negative of the Euler vector field $x^i\partial_{x^i}$ and $\Tr$ is the trace of endomorphisms. \\
Now there is a linear map
$$
\phi_2: \wedge^2 V^\ast \otimes V \to V^\ast
$$
by taking partial traces: for $\psi \in \wedge^2 V^\ast, v \in V$, we set
$$
\phi_2( \psi \otimes v) = -\iota_v(\psi).
$$
We now claim that (the constant extension of) $\phi_2$ completes \eqref{diag:parttr} to an anti-commutative square. Indeed: let $\{e_i\}_{i = 1}^n$ be a basis of $V$, with corresponding coordinates $\{x^i\}_{i=1}^n$, and let $\psi \otimes v \in \Gamma(\wedge^2 V^\ast \otimes V)$. Then
\begin{align*}
\partial_1(\phi_2(\psi \otimes v)) 
&= -\partial_1(\iota_v(\psi)) \\
&= x^i\iota_{e_i}\iota_v(\psi)\\
&= -x^i\iota_v\iota_{e_i}(\psi)\\
&= -\Tr(x^i\iota_{e_i}(\psi)\otimes v)\\
&= -\Tr(d_2(\psi\otimes v)).
\end{align*}
More generally, for $1\leq k \leq n$, we can define the anti-symmetrized partial trace map
$$
\phi_k: \wedge^k V^\ast \otimes V \to \wedge^{k-1} V^\ast.
$$
For $\alpha \in \wedge^k V^\ast, v \in V$, we set
$$
\phi_k(\alpha \otimes v) = (-1)^{k-1} \iota_v(\alpha).
$$
Observe that $\phi_1$ is the usual trace.\\
Note that the map $\partial_1: \Gamma(V^\ast) \to \Gamma(\mb R)$ as in \eqref{diag:parttr} of free $C^\infty(V)$-modules can be extended to obtain a cochain complex
\begin{equation}\label{eq:vanishideal}
\begin{tikzcd}
0 \arrow{r} & \Gamma(\wedge^n V^\ast) \arrow{r}{\partial_n}& \Gamma(\wedge^{n-1} V^\ast) \arrow{r}{\partial_{n-1}} &\Gamma(\wedge^{n-2} V^\ast) \arrow{r}{\partial_{n-2}} &\dots \arrow{r}{\partial_2}& \Gamma(V^\ast) \arrow{r}{\partial_1}& C^\infty(V) \arrow{r}& 0
\end{tikzcd}.
\end{equation}
The cochain complex is concentrated in negative degrees, with $C^\infty(V)$ being in degree $-1$. Note that by Lemma \ref{lem:koszulexact}, the complex  is exact in degrees $-2,\dots, -n-1$, as it is the truncation of \eqref{eq:koszexact2}. The following lemma describes the compatibility of $\phi$ with the respective differentials:
\begin{lem}
The map $\phi:(\Gamma(\wedge^\bullet V^\ast \otimes V),d_\bullet) \to (\Gamma( \wedge^{\bullet -1} V^\ast),\partial_\bullet)$ is a cochain map of degree $-1$, which is surjective in degrees $0,\dots, -(n-2)$, and an isomorphism in degree $-n+1$, where $d_\bullet$ is as in \eqref{eq:resgln}, and $\partial_\bullet$ is as in \eqref{eq:vanishideal}.
\end{lem}
\begin{proof}
Let $\alpha \otimes v \in \Gamma(\wedge^{k+1} V^\ast\otimes  V)$. We first show that $\phi$ anti-commutes with the respective differential:
\begin{align*}
    \partial_k(\phi_{k+1}(\alpha\otimes v)) &= \partial_k((-1)^{k}\iota_v(\alpha))\\
    &= (-1)^{k+1} x^i\iota_{e_i}(\iota_v(\alpha))\\
    &= (-1)^{k} x^i\iota_v(\iota_{e_i}(\alpha))\\
    &= -\phi_{k}(x^i\iota_{e_i}(\alpha) \otimes v)\\
    &= -\phi_{k}(d_{k+1}(\alpha \otimes v)).
\end{align*}
To see the surjectivity, pick a basis $\{e_i\}_{i =1}^n$ of $V$, and a dual basis $\{e^i\}_{i = 1}^n$ of $V^\ast$ such that $\mu = e^1\wedge\dots \wedge e^n$. For $k\in \{1,\dots, n\}$, a basis for $\wedge^{k-1} V^\ast$ is given by $\{e^{i_1}\wedge \dots \wedge e^{i_{k-1}}\mid 1\leq i_1 <\dots < i_{k-1} \leq n\}$. Given $e^{i_1} \wedge \dots \wedge e^{i_{k-1}}$, let $q\in \{1,\dots, n\}-\{i_1,\dots, i_{k-1}\}$. Then 
$$
\phi_k(e^{i_1}\wedge \dots \wedge e^{i_{k-1}}\wedge e^q \otimes e_q) = e^{i_1}\wedge \dots \wedge e^{i_{k-1}},
$$
where $q$ is \emph{not} summed over.
Further, under the identification 
\begin{align*}
    V&\to \wedge^n V^\ast \otimes V\\
    v &\mapsto e^1 \wedge \dots \wedge e^n \otimes v,
\end{align*}
$\phi_n$ is the map $V\to \wedge^{n-1}V^\ast$ given by contraction with the volume form $e^1\wedge \dots \wedge e^n$, which is an isomorphism.
\end{proof}
We use the properties of $\phi$ to construct a projective resolution for $\mc F_\mu$.
\begin{prop}\label{prop:ressl}
Let for $i = 1,\dots, n$ $K_i\subset \wedge^i V^\ast \otimes V$ be defined by $$K_i := \ker(\phi_i).$$The sequence 
\begin{equation}\label{diag:ressln}
\begin{tikzcd}
    0 \arrow{r} & \Gamma(\wedge^n V^\ast) \arrow{r}{d_n \phi_{n}^{-1}\partial_n} &\Gamma(K_{n-1})\arrow{r}{d_{n-1}} & \dots \arrow{r}{d_2} & \Gamma(K_1) \arrow{r}{\rho_\mu} &\mc F_\mu(V)\arrow{r}& 0
\end{tikzcd} 
\end{equation}
is exact, where $\rho_\mu = \left. \rho\right|_{\Gamma(K_1)}$. 
\end{prop}
\begin{proof}
Note that by definition of $\phi_1$, $K_1 = \mf {sl}(V)$, so $\rho_\mu$ is surjective by definition of $\mc F_\mu(V)$.\\
Let $i \in \{1,\dots, n-2\}$. Consider the following diagram with (anti)-commuting squares, where the middle and bottom rows are exact by Lemma \ref{lem:koszulexact}:
\[
\begin{tikzcd}
\Gamma(K_{i+2}) \arrow{r}{d_{i+2}} \arrow{d}&\Gamma(K_{i+1}) \arrow{r}{d_{i+1}} \arrow{d} & \Gamma(K_i) \arrow{r}{d_i}\arrow{d} & \Gamma(K_{i-1}) \arrow{d}\\
\Gamma(\wedge^{i+2} V^\ast \otimes V) \arrow{r}{d_{i+2}} \arrow{d}{\phi_{i+2}} & \Gamma(\wedge^{i+1} V^\ast \otimes V) \arrow{d}{\phi_{i+1}}\arrow{r}{d_{i+1}} & \Gamma(\wedge^i V^\ast \otimes V) \arrow{r}{d_i} \arrow{d}{\phi_i} & \Gamma(\wedge^{i-1} V^\ast \otimes V)\arrow{d}{\phi_{i-1}}\\
\Gamma(\wedge^{i+1} V^\ast ) \arrow{r}{\partial_{i+1}}&\Gamma(\wedge^{i} V^\ast)\arrow{r}{\partial_i} & \Gamma(\wedge^{i-1} V^\ast) \arrow{r}{\partial_{i-1}} & \Gamma(\wedge^{i-2} V^\ast)
\end{tikzcd}
\]
For exactness at $\Gamma(K_i)$, take $\chi\in \Gamma(K_i)$ such that $d_i(\chi) = 0$, where $d_1$ is understood to be $\rho_\mu$. Then by exactness of the middle row, there exists $\psi \in \Gamma(\wedge^{i+1} V^\ast \otimes V)$ such that $d_{i+1}(\psi) = \chi$. Now $\psi$ may not be an element of $\Gamma(K_{i+1})$, so we consider $\phi_{i+1}(\psi)$.\\
Note that
$$
\partial_{i}\phi_{i+1}(\psi) = -\phi_{i}(\partial_{i+1}(\psi)) = -\phi_i(\chi) = 0,
$$
so by exactness of $\eqref{eq:vanishideal}$ there exists $\tau \in \Gamma(\wedge^{i+1} V^\ast)$ such that
$$
\phi_{i+1}(\psi) = \partial_{i+1}(\tau). 
$$
Using surjectivity of $\phi_{i+2}$, lift $\tau$ to an element $\tilde{\tau}\in \Gamma(\wedge^{i+2} V^\ast \otimes V)$. Then 
$$
\phi_{i+1}(\psi + \partial_{i+2}(\tilde{\tau})) = \phi_{i+1}(\psi)- \partial_{i+1}(\tau) = 0,
$$
so $\psi+ \partial_{i+2}(\tilde{\tau}) \in K_{i+1}$, and 
$$
\partial_{i+1}(\psi+\partial_{i+2}(\tilde{\tau})) = \chi.
$$
For exactness at $\Gamma(K_{n-1})$, consider 
\[
\begin{tikzcd}
0 \arrow{r} & \Gamma(\wedge^n V^\ast) \arrow{r}{d_n(\phi_{n})^{-1} \partial_n} & \Gamma(K_{n-1}) \arrow{d} \arrow{r}{d_{n-1}} & \Gamma(K_{n-2})\arrow{d}\\
0 \arrow{r} & \Gamma(\wedge^n V^\ast \otimes V) \arrow{r}{d_n} \arrow{d}{\phi_n} & \Gamma(\wedge^{n-1} V^\ast \otimes V) \arrow{d}{\phi_{n-1}} \arrow{r}{d_{n-1}}& \Gamma(\wedge^{n-2} V^\ast \otimes V) \arrow{d}{\phi_{n-2}}\\
\Gamma(\wedge^n V^\ast) \arrow{r}{\partial_n} & \Gamma(\wedge^{n-1} V^\ast) \arrow{r}{\partial_{n-1}} & \Gamma(\wedge^{n-2} V^\ast) \arrow{r}{\partial_{n-2}} &\Gamma(\wedge^{n-3} V^\ast) 
\end{tikzcd}.
\]
Let $\xi \in \Gamma(K_{n-1})$ such that 
$$
d_{n-1}(\xi) = 0. 
$$
Then there exists $\eta \in \Gamma(\wedge^{n} V^\ast \otimes V)$ such that
$$
d_n(\eta) = \xi.
$$
As $\phi_{n}$ is an isomorphism, we have $\eta = \phi^{-1}_n(\phi_n(\eta))$. Moreover, we know that
$$
\partial_{n-1}(\phi_n(\eta)) = -\phi_{n-1}(d_{n}(\eta)) = -\phi_{n-1}(\xi) = 0,
$$
so
$$
\phi_{n}(\eta) = \partial_{n}(\pi)
$$
for some $\pi\in \Gamma(\wedge^n V^\ast)$. Consequently, $$ \xi = d_n (\phi_{n}^{-1}(\partial_{n}(\pi))).$$ 
Finally, exactness at $\Gamma(\wedge^n V^\ast)$ is clear.
\end{proof}
\subsubsection{The $L_\infty$-algebroid structure}\label{sec:bracketssl}
In this section, we will construct the $\linfty$-algebroid structure on the resolution \eqref{diag:ressln} of $\mc F_\mu(V)$.  As in most degrees the spaces involved in the resolution of $\mc F_\mu$ are contained in the spaces involved in the resolution of $\mc F_0$, we try to use the restriction of \eqref{eq:glnbracket}. The following lemma shows that this can be done:
\begin{lem}\label{lem:bracketrestr}
The bracket \eqref{eq:glnbracket} restricts to the subspaces $\Gamma(K_i)$.
\end{lem}
This gives us a hint on how to extend the bracket to \eqref{diag:ressln}: on the subcomplex given by the part up until degree $n-1$, it is given by \eqref{eq:glnbracket}. Note that there is no issue when $k_1 + k_2 - 1 = n$: since the bracket should land in $\Gamma(K_{n}) = 0$, we can unambiguously extend this definition when we replace $\Gamma(K_{n})$ by $\Gamma(\wedge ^n V^\ast)$.\\
For degree reasons and the Leibniz identity in a $L_\infty$-algebroid, we only have to specify what happens when we pair the constant extension of $X \in K_1 = \mf {sl}(V)$ with the constant extension of $\mu \in \wedge ^n V^\ast$. Due to the requirement that the differential is a derivation of the binary bracket, there is only one choice for this: We set
\begin{equation}\label{eq:bracketsln2}
[X, \mu] := 0 \in \Gamma(\wedge^nV^\ast).
\end{equation}
We then obtain:
\begin{prop}\label{prop:slbracket}
The binary operation defined by the restriction of \eqref{eq:glnbracket} on the spaces $\Gamma(K_i)$, together with the extension of \eqref{eq:bracketsln2} defines a dg-Lie algebroid structure on the resolution \eqref{diag:ressln} of $\mc F_\mu(V)$, which is a universal $L_\infty$-algebroid of $\mc F_\mu$, which is minimal at the origin.
\end{prop}
\subsection{Vector fields preserving the linear symplectic form}\label{sec:linsymp}
Next, we consider the symplectic Lie algebra. Given a vector space $V$ of even dimension $n$, and a non-degenerate skew-symmetric bilinear map $\omega:V\times V \to \mb R$, we consider the Lie subalgebra of $\mf {gl}(V)$ preserving $\omega$:
\begin{defn}
Let $(V,\omega)$ be a symplectic vector space. The \emph{symplectic Lie algebra} is the Lie subalgebra of $\mf {gl}(V)$ given by
\begin{equation}\label{eq:sympliealg}
    \mf{sp}(V,\omega):= \{A\in \mf {gl}(V)\mid \omega(Ax,y) + \omega(x,Ay) = 0\,\,\, \forall x,y \in V\},
\end{equation}
\end{defn}
By restricting the anchor $\rho$ to $\Gamma(\mf {sp}(V,\omega))$, we obtain a singular foliation 
$$
\mc F_\omega(V):= \rho(\Gamma(\mf {sp}(V,\omega)).
$$
In section \ref{sec:projressp} we construct a projective resolution of $\mc F_\omega(V)$ (Proposition \ref{prop:spres}). In section \ref{sec:bracketssp} we construct a part of the $L_\infty$-algebroid structure. We give an expression for the binary bracket (Proposition \ref{prop:spbracket}), depending on a choice of left inverse $r^\omega$ of an injective cochain map. This bracket does not satisfy the Jacobi identity, so we give an expression for the ternary bracket, which serves as a contracting homotopy for the Jacobiator. In appendix \ref{appendix1} we investigate whether $r^\omega$ can be chosen to be a cochain map in some degrees, which would simplify the binary bracket. We show that when $\dim (V) = 4$, $r^\omega$ can not be chosen as a cochain map in any degree (Proposition \ref{prop:nosplitdiff}).
\subsubsection{The projective resolution}\label{sec:projressp}
As before, we first construct the projective resolution of the foliation $\mc F_\omega = \rho(\Gamma(\mf {sp}(V,\omega)))$ on $V$. The starting point is the same as for $\mf {sl}(V)$, but the rest of the approach will be quite different, as the analog of the map $\phi$ is not surjective in negative degrees. 
First consider the map $\phi_1^\omega:\mf {gl}(V) \to \wedge ^2 V^\ast$ given by
$$
\phi_1^\omega(A) = A\cdot \omega
$$
for $A\in \mf {gl}(V)$, where for $x,y \in V$,
$$
(A\cdot\omega)(x,y) = \omega(Ax,y) + \omega(x,Ay).
$$
The next step is to extend this map to the entire dg-Lie algebra $\Gamma(\wedge^\bullet V^\ast \otimes V)$, as in \eqref{eq:resgln} and \eqref{eq:glnbracket}. This immediately raises the question what the codomain should be. We define for $p = 1,\dots, n$ 
$$
\phi^\omega_p:\wedge^p V^\ast \otimes V \to \wedge^{p-1} V^\ast \otimes \wedge^2V^\ast
$$
by
$$
\phi^\omega_p(\alpha\otimes X) := (-1)^{p-1}\iota_{e_i}(\alpha)\otimes e^i\wedge \iota_X\omega,
$$
where $\{e_i\}_{i = 1}^n$ and $\{e^i\}_{i =1}^n$ are dual bases of $V$ and $V^\ast$ respectively. When viewing the domain and codomain of $\phi_p^\omega$ as $\Hom(\wedge^p V,V)$ and $\Hom(\wedge^{p-1} V, \wedge^2 V^\ast)$ respectively, the map $\phi^\omega_p$ can equivalently be described as 
$$
\phi^\omega_p(\psi)(v_1,\dots,v_{p-1}) = \psi(v_1,\dots,v_{p-1},-)\cdot\omega
$$
for $\psi \in \Hom(\wedge^p V,V)$, $v_1,\dots, v_{p-1}\in V$.\\
We now equip the graded $C^\infty(V)$-module $\Gamma(\wedge^{\bullet} V^\ast \otimes \wedge^2 V^\ast)$
with the differential
$$
\partial_p: \Gamma(\wedge^p V^\ast \otimes \wedge^2 V^\ast) \to \Gamma(\wedge^{p-1} V^\ast \otimes \wedge^2 V^\ast)
$$
given by
$$
\partial_p(\alpha\otimes \tau) := -x^i\iota_{e_i}(\alpha)\otimes \tau.
$$
Here the grading is chosen as
$$
\Gamma(\wedge^\bullet V^\ast \otimes \wedge^2 V^\ast)^p = \Gamma(\wedge^{p+1} V^\ast \otimes \wedge^2 V^\ast).
$$
Finally, there is an action of the dg-Lie algebra $\Gamma(\wedge^\bullet V^\ast \otimes V)$ on $\Gamma(\wedge^\bullet V^\ast \otimes \wedge^2 V^\ast)$: for $\alpha\otimes X \in \wedge^p V^\ast \otimes V, \beta \otimes \tau \in \wedge ^q V^\ast \otimes \wedge^2 V^\ast$, we set
\begin{equation}\label{eq:liealgactsympl}
(\alpha\otimes X) \cdot (\beta\otimes \tau) := (-1)^{p-1}\alpha\iota_X(\beta)\otimes \tau + \iota_{e_i}(\alpha)\beta\otimes e^i\wedge \iota_X(\tau) 
\end{equation}
for constant sections, and extend it to non constant sections using the Leibniz rule with respect to the anchor of $\Gamma(\mf{gl}(V))$. 
Note that by multiplying by $-(-1)^{(p-1)(q-1)}$, we can turn this into a right action. 
We have the following lemma summarizing the properties of the data above, of which the proof is a direct computation.
\begin{lem}\label{lem:propspo}
\begin{itemize}
    \item[]
    \item[i)] $\phi_1^\omega$ is surjective, and $\ker(\phi_1^\omega) = \mf {sp}(V,\omega)$.
    \item [ii)] For $k\geq 2$, $\phi_k^\omega$ is injective. In particular, $\phi^\omega_2$ is an isomorphism.
    \item[iii)] $\phi^\omega$ is a cochain map of degree $-1$, i.e. we have $$ \phi^\omega_p d_{p+1} + \partial_p \phi^\omega_{p+1} = 0.$$
    \item[iv)] The operation defined in \eqref{eq:liealgactsympl} is a dg-Lie algebra action. Consequently, there is a dg-Lie algebra structure on $\Gamma(\wedge^\bullet V^\ast \otimes V) \oplus \Gamma(\wedge^\bullet V^\ast \otimes \wedge^2 V^\ast)$ encoding this action. \\For $p \geq 0$, the degree $p$-part is given by $\Gamma(\wedge^{p+1} V^\ast \otimes V) \oplus \Gamma(\wedge^{p+1} V^\ast \otimes \wedge^2 V^\ast)$, and the differential is given by $d+ \partial$.
    \item[v)] $\phi^\omega$ is a derivation of the bracket on $\Gamma(\wedge^\bullet V^\ast \otimes V) \oplus \Gamma(\wedge^\bullet V^\ast \otimes \wedge^2 V^\ast)$: for $\alpha\otimes X \in \Gamma(\wedge^pV^\vee \otimes V), \beta \otimes Y \in \Gamma(\wedge^q V^\vee \otimes V)$, we have the equality 
    \begin{equation}\label{eq:phideriv}
        \phi_{p+q-1}^\omega([\alpha\otimes X,\beta\otimes Y]) = [\phi_p^\omega(\alpha\otimes X), \beta \otimes Y] + (-1)^{p-1}[\alpha\otimes X, \phi^\omega_q(\beta\otimes Y)].
    \end{equation}
\end{itemize}
\end{lem}
\begin{cor}
By property iii), the differentials $d_\bullet$ and $\partial_\bullet$ restrict and descend to the kernel and cokernel of $\phi^\omega$ respectively. We denote the differential induced by $\partial_\bullet$ on $\Gamma(\coker(\phi^\omega))$ by $\bar{\partial_\bullet}$.
\end{cor}
Using these properties, we can construct a projective resolution:
\begin{prop}\label{prop:spres}
For $i = 3,\dots, n+1$, let $C_i$ be defined as $C_i := \coker(\phi_i^\omega)$. The sequence
\begin{equation}\label{diag:resspo}
\begin{tikzcd}
0\arrow{r}& \Gamma(C_{n+1})  \arrow{r}{\bar{\partial}_n} &\Gamma(C_n)\arrow{r}{\bar{\partial}_{n-1}}& \dots \arrow{r}{\bar{\partial}_3} & \Gamma(C_3) \arrow{r}{d_2(\phi^\omega_2)\inv \partial_2}& \Gamma(\mf{sp}(V,\omega)) \arrow{r} &\mc F_\omega(V)\arrow{r}& 0 
\end{tikzcd},
\end{equation}
is exact. Here $\phi_{n+1}^\omega:0 \to \Gamma(\wedge^n V^\ast \otimes \wedge^2 V^\ast)$ is understood to be the zero map.
\end{prop}
\begin{proof}
We start by proving exactness at $\Gamma(C_p)$ for $p = 4,\dots, n+1$. Consider the diagram
\[
\begin{tikzcd}
\Gamma(\wedge^{p+1} V^\ast \otimes V) \arrow{r}{d_{p+1}} \arrow{d}{\phi_{p+1}^\omega}& \Gamma(\wedge^p V^\ast \otimes V) \arrow{r}{d_{p}}\arrow{d}{\phi_p^\omega}& \Gamma(\wedge^{p-1} V^\ast \otimes V) \arrow{d}{\phi^\omega_{p-1}} \arrow{r}{d_{p-1}} & \Gamma(\wedge^{p-2}V^\ast \otimes V) \arrow{d}{\phi_{p-2}^\omega}\\
\Gamma(\wedge^{p} V^\ast \otimes \wedge^2 V^\ast) \arrow{d}\arrow{r}{\partial_p} & \Gamma(\wedge^{p-1} V^\ast \otimes \wedge^2 V^\ast) \arrow{r}{\partial_{p-1}} \arrow{d}& \Gamma(\wedge^{p-2} V^\ast \otimes \wedge^2 V^\ast) \arrow{r}{\partial_{p-2}} \arrow{d} & \Gamma(\wedge^{p-3} V^\ast \otimes \wedge^2 V^\ast)\arrow{d}\\
\Gamma(C_{p+1}) \arrow{r}{\overline{\partial_{p}}} & \Gamma(C_p) \arrow{r}{\overline{\partial_{p-1}}}& \Gamma(C_{p-1}) \arrow{r}{\overline{\partial_{p-2}}} & \Gamma(C_{p-2}) 
\end{tikzcd}
\]
For $\tau\in \Gamma(\wedge^{p-1} V^\ast \otimes \wedge^2 V^\ast)$, assume that there exists $X\in \Gamma(\wedge^{p-1} V^\ast \otimes V)$ such that
$$
\partial_{p-1}(\tau) = \phi_{p-1}^\omega(X).
$$
Then 
$$
\phi_{p-2}^\omega(d_{p-1}(X)) = - \partial_{p-2}(\phi_{p-1}^\omega(X)) = 0,
$$
and by injectivity of $\phi_{p-1}^\omega$, it follows that $d_{p-1}(X) = 0$. By exactness of \eqref{eq:resgln} we find $X = d_{p}(Y) $. Now
$$
\partial_{p-1}(\tau + \phi_{p}^\omega(Y)) = \phi_{p-1}^\omega(X) - \phi_{p-1}^\omega(X) = 0,
$$
so by exactness of \eqref{eq:koszexact2}, with $W = \wedge^2 V^\ast$, we find that there exists $\mu \in \Gamma(\wedge^p V^\ast \otimes \wedge^2V^\ast)$ such that $$\tau = \partial_p(\mu) - \phi_{p}^\omega(Y),$$ showing that the class of $\tau$ modulo the image of $\phi^\omega$ is a coboundary.\\

For exactness at $\Gamma(C_3)$, assume that for $\tau\in \Gamma(\wedge^2V^\ast \otimes \wedge^2V^\ast)$ such that
$$
d_2(\phi_2^\omega)^{-1}\partial_2(\tau) = 0.
$$
Then there exists $X\in \Gamma(\wedge^3 V^\ast \otimes V)$ with
$$
(\phi_2^\omega)^{-1}\partial_2 (\tau) = d_3(X),
$$
or 
$$
\partial_2(\tau + \phi_2^\omega(X)) = 0, 
$$
which implies that there exists $\mu\in \Gamma(\wedge^3 V^\ast \otimes \wedge^2 V^\ast)$ such that
$$
\partial_3(\mu) = \tau+\phi_2^\omega(X),
$$
so the class of $\tau$ module the image of $\phi^\omega$ is a coboundary.\\
Finally, by definition of $\mc F_\omega$ the anchor restricted to $\mf {sp}(V,\omega)$ is surjective, so it suffices to show that the kernel is precisely $(d_2(\phi_2^\omega)^{-1}\partial_2)(\Gamma(\wedge^2V^\ast \wedge^2 V^\ast)$. Let $A\in \Gamma(\mf{sp}(V,\omega))$ such that $$\rho_\omega(A) = 0.$$ Then $A = d_2(X)$ for some $X\in \Gamma(\wedge^2 V^\ast \otimes V)$. As $\phi_2^\omega$ is an isomorphism, we have
$$
A = d_2(\phi_2^\omega)^{-1} \phi_2^\omega(X).
$$
As
$$
\partial_1\phi_2^\omega(X) =- \phi_1^\omega(d_2(X)) = \phi_1^\omega(A) = 0,
$$
it follows that $\phi_2^\omega(X) = \partial(\mu)$ for some $\mu \in \Gamma(\wedge^2 V^\ast \otimes \wedge^2 V^\ast)$, and 
$$
A = d_2(\phi_2^\omega)^{-1}\partial_2(\mu), 
$$
concluding the proof.
\end{proof}
\subsubsection{The (partial) $L_\infty$-algebroid structure}\label{sec:bracketssp}
In this section we construct part of an $L_\infty$-algebroid structure on the resolution \eqref{diag:resspo}. Due to the algebraic structures present, there is a canonical Lie bracket $\{-,-\}$ on the graded $C^\infty(V)$-module underlying the resolution \eqref{diag:resspo}. It is given by
$$
\{A,B\} = [A,B]
$$
for $A,B\in \Gamma(\mf {sp}(V,\omega))$, where $[-,-]$ is the usual bracket, and
$$
\{A,\omega_k + \text{im}(\phi^\omega_{k+1})\} = [A, \omega_k] + \text{im}(\phi^\omega_{k+1})
$$
for $A\in \Gamma(\mf {sp}(V,\omega)), \omega_k \in \Gamma(\coker(\phi^\omega_{k+1}))$. Here the bracket $[-,-]$ on the right hand side is the semi-direct product bracket as described in Lemma \ref{lem:propspo}iv).
\begin{rmk}
One way to see that $\{-,-\}$ is well-defined, is by forgetting the differentials $d_\bullet$ and $\partial_\bullet$, and viewing $(\Gamma(\wedge^\bullet V^\ast) \oplus \Gamma(\wedge^\bullet V^\ast \otimes \wedge^2 V^\ast), \phi^\omega)$ as a dg-Lie algebra. The resolution \eqref{diag:resspo} of $\mc F_\omega(V)$ is then precisely the cohomology of this dg-Lie algebra. Consequently, the bracket equips the graded module with a graded Lie algebra structure.
\end{rmk}
When at least one of the entries of $\{-,-\}$ has degree 0, the differential in \eqref{diag:resspo} is a derivation of $\{-,-\}$. However, for elements $\omega_1,\omega_2 \in \Gamma(\coker(\phi^\omega_{3}))$, for the differential in \eqref{diag:resspo} to be a derivation of the bracket, the equation
\begin{equation}\label{eq:dercheck}
[d_2(\phi_2^\omega)^{-1}\partial_2 \omega_1, \omega_2] - [\omega_1,d_2(\phi_2^\omega)^{-1}\partial_2 \omega_2] = \overline{\partial_3}[\omega_1,\omega_2] = 0 \in \Gamma(C_4)
\end{equation}
must hold. This means that the expression \eqref{eq:dercheck} must lie in the image of $\phi_4^\omega$, but one can check that this is not the case: the binary operation $\{-.-\}$ therefore does not equip the resolution \eqref{diag:resspo} with a $L_\infty$-algebroid structure, as the differential is not a derivation of the binary bracket.\\
To rectify this, we modify $\{-.-\}$ to obtain a new binary operation $\llbracket-,- \rrbracket$ on the resolution \eqref{diag:resspo}, for which the differential is a derivation.\\
Before we define this binary operation, we make a choice of left inverse of the map 
$$
\phi_p^\omega:\Gamma(\wedge^p V^\ast \otimes V) \to \Gamma(\wedge^{p-1} V^\ast \otimes \wedge^2 V^\ast)
$$
for $p\geq 2$.
Define
$$
r_p^\omega: \wedge^{p} V^\ast \otimes \wedge^2 V^\ast \to \wedge^{p+1} V^\ast \otimes V
$$
by 
$$
r_p^\omega(\omega_p \otimes \tau) = \left(\frac{1}{p+1} \omega_p \wedge \iota_{e_i}(\tau) - \frac{(-1)^p}{p(p+1)} \iota_{e_i}(\omega_p)\wedge \tau\right)\otimes \omega^{-1}(e^i) \in \wedge^{p+1} V^\ast \otimes V
$$
for $\omega_p\in \wedge^pV^\ast, \tau \in \wedge^2 V^\ast$. Further, $\{e_i\}_{i=1}^n, \{e^i\}_{i=1}^n$ are dual bases for $V$ and $V^\ast$ respectively, and $\omega^{-1}: V^\ast \to  V$ is the inverse of the contraction map $\omega:V \to V^\ast$. The proof of the following lemma is a straightforward computation:
\begin{lem}\label{lem:rprop}
\begin{itemize} 
    \item[]
    \item[i)] For $k\geq 2$, $r_k^\omega$ intertwines the $\mf {sp}(V,\omega)$-action on $\Gamma(\wedge^k V^\ast \otimes \wedge^2 V^\ast)$ and $\Gamma(\wedge^{k+1} V^\ast \otimes V)$.
    \item[ii)] For $k\geq 2$, $$r_k^\omega\circ \phi_{k+1}^\omega = \text{id}_{\Gamma(\wedge^{k+1} V^\ast \otimes V)}.$$
\end{itemize}
\end{lem}
We can now give an expression for the binary operation $\llbracket-,-\rrbracket$ for which the differential of the resolution \eqref{diag:resspo} is a derivation, providing the binary bracket for an $L_\infty$-algebroid structure on the resolution.
\begin{prop}\label{prop:spbracket}
When at least one entry of $\llbracket -,-\rrbracket$ has degree $0$, we set
$$
\llbracket -,-\rrbracket = \{-,-\}.
$$
Now let $p,q\geq 2$. For $\omega_p \in \Gamma(\wedge^p V^\ast\otimes \wedge^2 V^\ast), \omega_q \in \Gamma(\wedge^q V^\ast \otimes \wedge^2 V^\ast)$, set
\begin{equation}\label{eq:spbracket2}
\llbracket \omega_p,\omega_q\rrbracket := [r_{p-1}^\omega\partial_pP_p(\omega_p), P_q(\omega_q)] + [P_p(\omega_p),r^\omega_{q-1}\partial_qP_q(\omega_q)] \mod \text{im}(\phi_{p+q}^\omega) \in \Gamma(\coker(\phi^\omega_{p+q})).
\end{equation}
Here $P_p: \Gamma(\wedge^p V^\ast \otimes \wedge^2 V^\ast) \to \Gamma(\ker(r_p^\omega))$ is the projection $$P_p = \text{id}- \phi_{p+1}^\omega \circ r_{p}^\omega,$$ and $[-,-]$ on the right hand side is the semi-direct product bracket as described in Lemma \ref{lem:propspo}iv).\\
Then the differential of \eqref{diag:resspo} is a derivation of $\llbracket-,-\rrbracket$.
\end{prop}
\begin{proof}
The proof is a direct computation using Lemma \ref{lem:propspo}iv) and Lemma \ref{lem:rprop}ii).
\end{proof}
The natural question is now: Does $\llbracket-,-\rrbracket$ satisfy the Jacobi identity? \\

To address this, we distinguish two cases. We first compute the Jacobiator when at least one of the entries has degree 0, and then when all the entries have negative degree.
\begin{itemize}
    \item[-] For $A,B,C\in \Gamma(\mf{sp}(V,\omega))$, the Jacobiator of $\llbracket-,-\rrbracket$ is the Jacobiator of the Lie algebroid $\Gamma(\mf {sp}(V,\omega)$, which vanishes.
    \item[-] For $A,B \in \Gamma(\mf{sp}(V,\omega))$, $\omega_k\in \Gamma(\coker(\phi^\omega_{k+1}))$, the Jacobiator being zero is equivalent to $\coker(\phi^\omega_{k+1})$ being an $\mf {sp}(V,\omega)$-representation.
    \item[-] For $A \in \Gamma(\mf{sp}(V,\omega))$, $\omega_k \in \Gamma(\coker(\phi^\omega_{k+1}), \omega_l \in \Gamma(\coker(\phi^\omega_{l+1})$, the Jacobiator being zero is equivalent to the $\mf {sp}(V,\omega)$-action on $\Gamma(\coker(\phi^\omega_{j}))$ being a derivation of $\llbracket-,-\rrbracket$ restricted to negative degrees. This is the case because $r_k$ intertwines the $\mf {sp}(V,\omega)$-actions on $\Gamma(\wedge^k V^\ast \otimes \wedge^2 V^\ast)$ and $\Gamma(\wedge^{k+1} V^\ast \otimes \wedge^2 V^\ast)$. 
\end{itemize}
Consequently, the Jacobiator vanishes when at least one entry has degree 0. \\
Now let $k,l,m\geq 2$, and $\omega_k \in \Gamma(\coker(\phi^\omega_{k+1}))$, $\omega_l \in \Gamma(\coker(\phi^\omega_{l+1})), \omega_m \in \Gamma(\coker(\phi^\omega_{m+1}))$. A lengthy computation shows that the Jacobiator
$$
\llbracket\llbracket\omega_k,\omega_l\rrbracket,\omega_m\rrbracket + (-1)^{(k-1)(l+m)} \llbracket\llbracket\omega_l,\omega_m\rrbracket,\omega_k\rrbracket + (-1)^{(m-1)(k+l)}
\llbracket\llbracket\omega_m,\omega_k\rrbracket,\omega_l\rrbracket
$$
does not vanish, but is equal to
\begin{equation}\label{eq:jacobneg3}
\overline{\partial}(\llbracket-,-,-\rrbracket)(\omega_k,\omega_l,\omega_m) = \overline{\partial}\llbracket\omega_k,\omega_l,\omega_m\rrbracket + \llbracket\overline{\partial}(\omega_k),\omega_l,\omega_m\rrbracket + (-1)^{k-1} \llbracket\omega_k,\overline{\partial}(\omega_l),\omega_m\rrbracket + (-1)^{k+l} \llbracket\omega_k,\omega_l,\overline{\partial}(\omega_m)\rrbracket,
\end{equation}
where $\llbracket\omega_k,\omega_l,\omega_m\rrbracket$ is given by the class of 
\begin{equation}\label{eq:terbrack}
 [r_{k+l-1}\widehat{\llbracket\omega_k,\omega_l\rrbracket},P_m(\omega_m)]  + (-1)^{(k-1)(l+m)} [r_{k+l-1}\widehat{\llbracket\omega_l,\omega_m\rrbracket},P_k(\omega_k)]  + (-1)^{(m-1)(k+l)}
[r_{k+l-1}\widehat{\llbracket\omega_m,\omega_k\rrbracket},P_l(\omega_l)]
\end{equation}
modulo the image of $\phi^\omega_{k+l+m-1}$.
and
$$
\widehat{\llbracket\omega_k,\omega_l\rrbracket} = [r_{k-1}\partial_kP_k(\omega_k), P_l(\omega_l)] + [P_k(\omega_k),r_{l-1}\partial_lP_l(\omega_l)] \in \Gamma(\wedge^{k+l-1} V^\ast \otimes \wedge^2 V^\ast).
$$
We recognize equation \eqref{eq:jacobneg3} as a contracting homotopy for the Jacobiator: consequently, $-\llbracket-,-,-\rrbracket$ is a ternary operation satisfying the higher Jacobi identity an $L_\infty$-algebroid must satisfy. 

In particular, this does not equip the complex \eqref{diag:resspo} with the structure of a dg-Lie algebroid as in the case of $\mf {gl}(V)$, $\mf {gl}(V,W)$, and $\mf {sl}(V)$. As this structure is only unique up to $L_\infty$-algebroid homotopy, this does of course not exclude the possibility that there exists a dg-Lie algebroid structure inducing the foliation $\mc F_\omega$.\\

In appendix \ref{appendix1}, we investigate to what extent $r^\omega$ can be chosen to (anti)-commute with the differentials, as this would simplify both the binary and ternary bracket.
\begin{rmk}\label{rmk:whendgla}
\begin{itemize}
    \item[]
    \item[-] For degree reasons, the operation $\llbracket-,-,-\rrbracket$ vanishes when $\dim V \leq 4$. This means that when $\dim V = 2$ or $\dim V = 4$, the foliation $\mc F_{\omega}$ does admit a universal $L_\infty$-algebroid with only a unary and binary bracket. Of course, for $\dim V = 2$, $\mf {sp}(V,\omega) = \mf {sl}(V)$, for which it was already known that a dg-Lie algebroid structure structure exists.
    \item[-] When $\dim V = 6$, the unary, binary and ternary bracket determine the full $L_\infty$-algebroid structure.
\end{itemize}
\end{rmk}

\section{Higher-dimensional leaves}\label{sec:hdleaf}
In this section we address question 2) of section \ref{a}. Instead of considering foliations on a vector space with linear generators, we consider foliations $\mc F$ on vector \emph{bundles} $\pi:E\to L$ which are generated by \emph{fiberwise} linear vector fields such that the zero section is a leaf. \\

For $x\in L$ the fibers $E_x = \pi^{-1}(\{x\})$ of $\pi:E\to L$ are transverse to $L$, the foliation $\mc F$ restricts to the fibers (\cite[Proposition 1.10]{holgpd}). We consider foliations for which the restriction $\left. \mc F\right|_{E_x}$ coincides with one of the examples in section \ref{sec:0dimleaf}. All results given in section \ref{sec:0dimleaf} will carry over, although in order to define the analogue of $\mc F_\mu$ and $\mc F_\omega$ we need the existence of non-vanishing sections of $\wedge^{\rk(E)} E^\ast$ and $\wedge^2 E^\ast$ respectively.\\

In section \ref{sec:zerosec}, we consider the foliation of all vector fields on the vector bundle $E$, for which the restriction to the zero section is tangent to the zero section. This foliation consists of \emph{all} fiberwise linear vector fields and is the analogue of $\mc F_0$ in section \ref{sec:vf00}.\\

In section \ref{sec:subbundle}, we consider the foliation of all fiberwise linear vector fields on $E$ tangent to a subbundle $D$, which is the analogue of $\mc F_W$ in section \ref{sec:subsp}.\\

In section \ref{sec:vol1}, we assume that $E$ is orientable, and consider the foliation of all fiberwise linear vector fields on $E$ which preserve a non-vanishing (on $L$) section $\mu \in \Gamma(\wedge^{\rk(E)} E^\ast)$, which is the analogue of $\mc F_\mu$ in section \ref{sec:vol0}.\\

In section \ref{sec:sympleaf}, we assume that $E\to L$ is a \emph{symplectic} vector bundle with non-degenerate $\omega \in \Gamma(\wedge^2 E^\ast)$, and consider the foliation of all fiberwise linear vector fields on $E$ which preserve $\omega$, which is the analogue of $\mc F_\omega$ in section \ref{sec:linsymp}.
\subsection{Vector fields tangent to the zero section}\label{sec:zerosec}
To generalize section \ref{sec:vf00}, we need a generalization of the Lie algebra $\mf{gl}(V)$ for a vector space $V$ to a vector bundle. Let $L$ be a smooth manifold, and let $\pi:E\to L$ be a (real) vector bundle of rank $n$. One way to generalize $\mf {gl}(V)$ would be to use the Lie algebra bundle $\End(E)$. This is however not the right thing to consider: although it acts infinitesimally on $E$, for $x\in L$ the leaves of the foliation are given by $0_x\in E_x$ and the connected components of $E_x\setminus{0_x}$. We are however interested in the situation where the zero section is a leaf, and the transverse foliation at a point of the zero section is given by \eqref{eq:0inorigin}. This is the case when dealing with a linearizable foliation around an embedded codimension $n$ leaf with isotropy Lie algebra bundle $\End(E)$ acting on the $n$-dimensional fiber of the normal bundle by all vector bundle maps $E\to E$.\\
First recall that there are two distinguished classes of smooth functions on a vector bundle $E$. The fiber-wise constant maps, given by the image of $\pi^\ast: C^\infty(L) \to C^\infty(E)$, and the fiber-wise linear ones, given by $\Gamma(E^\ast)$.\\
Now let
$$
\mf X_{lin}(E) = \{X \in \mf X(E) \mid X (\pi^\ast(C^\infty(L)))\subset \pi^\ast(C^\infty(L)), \,X(\Gamma(E^\ast))\subset \Gamma(E^\ast)\}
$$
be the set of vector fields preserving the fiberwise constant and fiberwise linear functions.
First off, we note that $\mf X_{lin}(E)$ is isomorphic to the sections of a transitive Lie algebroid over $L$ (see \cite[Theorem 1.4]{cdoksmk}).
\begin{lem}\label{lem:linvf}
\begin{itemize}
\item[]
\item[i)] There is a short exact sequence of $C^\infty(L)$-modules
\[
\begin{tikzcd}
0 \arrow{r}& \Gamma(End(E)) \arrow{r}{a} & \mf X_{lin}(E) \arrow{r}{\rho} & \mf X(L) \arrow{r}& 0.
\end{tikzcd}
\]
Here $\rho$ is the restriction of a vector field to the subalgebra $\pi^\ast(C^\infty(L))$ and 
$$
a(A)(f) = \left.\frac{d}{dt}\right|_{t = 0}f\circ exp(-tA),
$$
which is a fiberwise extension of the identification of linear vector fields on a vector space with the endomorphisms on the vector space.
\item[ii)]$\mf X_{lin}(E)$ is a finitely generated projective $C^\infty(L)$-module. So there exists a vector bundle $\mf{gl}(E)$ such that $\mf X_{lin}(E) = \Gamma(\mf {gl}(E))$.
\item[iii)] $\mf X_{lin}(E)$ is closed under the Lie bracket of vector fields.
\item[iv)] The triple $(\mf {gl}(E), \rho, [-,-])$ is a Lie algebroid.
\end{itemize}
\end{lem}
\begin{proof}
i) is a local computation, ii) follows from i), and iii) and iv) are immediate. 
\end{proof}
We can now construct a singular foliation on $E$ for which $L$ is a leaf: let 
$$
\mc F_L(E) = \{X\in \mf X(E) \mid \left. X \right|_{L} \in \mf X(L)\}. 
$$
\begin{lem}\label{lem:tangenttol}
$\mc F_L(E) = \text{Im}(C^\infty(E) \otimes_{C^\infty(L)} \mf X_{lin}(E) \stackrel{m}{\to} \mf X(E))$, where $m$ is the natural multiplication map.
\end{lem}
\begin{proof}
Follows from a local computation.
\end{proof}
\begin{rmk}
Note that $C^\infty(E) \otimes_{C^\infty(L)} \mf X_{lin}(E) = \Gamma(\pi^\ast (\mf{gl}(E)))$, which are the sections of the action Lie algebroid corresponding to the natural action of $\mf X_{lin}(E)$ on $E$, its anchor is $m$, while the bracket is given by 
$$
[f\otimes X,g \otimes Y] = fg \otimes [X,Y] + fX(g)\otimes Y - gY(f) \otimes X
$$
for $f,g \in C^\infty(E)$, $X,Y \in \mf X_{lin}(E)$.
\end{rmk}
\subsubsection{The projective resolution}
Lemma \ref{lem:tangenttol} now gives a first step in the resolution of $\mc F_L(E)$: 
\[
\begin{tikzcd}
\Gamma(\pi^\ast(\mf {gl}(E))) \arrow{r}{m} &\mc F_L(E) \arrow{r} & 0.
\end{tikzcd}
\]
Observe that $m$ is not injective! However, the kernel can be explicitly described. As a first step, we show that the kernel only affects the direction transverse to the leaf.
\begin{lem}\label{lem:kerver}
$$\ker(m) \subset \Gamma(\pi^\ast(End(E)))\subset \Gamma(\pi^\ast(\mf{gl}(E))).$$
\end{lem}
The following argument is thanks to Marco Zambon.
\begin{proof}
There is a commutative diagram of $C^\infty(E)$-modules given by
\[
\begin{tikzcd}
\Gamma(\pi^\ast(\mf {gl}(E))) \arrow{r}{m}\arrow{rd}{\text{Id}\otimes \rho} & \mf X(E) \arrow{d}{d\pi}\\
& \Gamma(\pi^\ast(TL))
\end{tikzcd}.
\]
The statement now follows from Lemma \ref{lem:linvf}i).
\end{proof}
Viewing $End(E)$ as $E^\ast \otimes E$, we can write down a complex analogous to \eqref{eq:resgln}.
\begin{equation}\label{diag:resgle}
\begin{tikzcd}
0\arrow{r}& \Gamma(\pi^\ast(\wedge^n E^\ast \otimes E)) \arrow{r}{d_n} & \dots \arrow{r}{d_3} & \Gamma(\pi^\ast(\wedge^2E^\ast \otimes E)) \arrow{r}{d_2}& \Gamma(\pi^\ast(\mf {gl}(E))) \arrow{r}& \mc F_L(E)\arrow{r}&0,
\end{tikzcd}
\end{equation}
where 
$$
d_k: \Gamma(\pi^\ast(\wedge^k E^\ast \otimes E)) \to \Gamma(\pi^\ast(\wedge^{k-1} E^\ast \otimes E) )
$$
is given by
\begin{equation}\label{eq:contreuler}
d_k(\alpha \otimes e) = y^i \iota_{e_i}(\alpha)\otimes e,
\end{equation}
for $\alpha \otimes e \in \Gamma( \pi^\ast(\wedge^k (E^\ast\otimes E)))$, $\{e_i\}_{i=1}^n$ a local frame of $E$, and $\{y^i\}_{i =1}^n$ the corresponding linear coordinates (note that this does not depend on the choice of frame and defines a global section $\epsilon \in \Gamma(\pi^\ast E)$). We then have:
\begin{prop}\label{prop:resgle}
The complex \eqref{diag:resgle} is exact. 
\end{prop}
\begin{proof}
Proving exactness is completely analogous to the case considered in section \ref{sec:vf00}: it suffices to pick an open cover of $L$ over which $E$ trivializes (so $\pi^\ast(E)$ trivializes over the preimages of this cover), and show exactness over this open cover. But for trivial bundles the result is equivalent to the exactness of \eqref{eq:resgln}.
\end{proof}
\subsubsection{The $L_\infty$-algebroid structure}
We claim that we can again find a dg-Lie algebroid structure on the resolution \eqref{diag:resgle}. Note that for degrees $-1,\dots, -n+1$, the involved spaces are simply the fiberwise extensions of \eqref{eq:resgln}, so we take the fiberwise extension of \eqref{eq:glnbracket}. To incorporate $\mf X_{lin}(E)$ inside into this, we recall the following:
\begin{lem}\label{lem:cdoact}
The action of $\mf X_{lin}(E)$ on $\Gamma(E^\ast)$ extends to $\Gamma(E)$, all tensor, wedge and symmetric products and their pullbacks to $E$.  
\end{lem}
\begin{proof}
Recall that an action of $\mf X_{lin}(E)$ on a vector bundle $F$ is a flat $\mf{gl}(E)$-connection on the vector bundle $F$. As the action on $\Gamma(E^\ast)$ is equivalent to a $\mf{gl}(E)$-connection on $E^\ast$, one can dualize this connection, and extend it via the Leibniz rule to tensor powers. \\
Finally, to extend the action to the pullback, we recall that $\Gamma(\pi^\ast (E^\ast)) = C^\infty(E) \otimes_{C^\infty(L)} \Gamma(E^\ast)$, and that both factors have a natural action of $\mf X_{lin}(E)$.\\
Setting for $g\otimes X \in C^\infty(E)\otimes_{C^\infty(L)} \mf X_{lin}(E)$, $f \otimes \alpha \in \Gamma(\pi^\ast(E^\ast))$
$$
(g\otimes X)\cdot(f \otimes \alpha) = gX(f) \otimes \alpha + gf \otimes X(\alpha).
$$
Since duals and tensors commute with pullbacks, the result follows.
\end{proof}
Using these actions, we can describe a dg-Lie algebroid structure on the resolution \eqref{diag:resgle}:
\begin{prop}\label{prop:glebracket}
The complex \eqref{diag:resgle} carries a dg-Lie algebroid structure, where the binary bracket is given by the analogue of equation \eqref{eq:glnbracket} on elements of degree $-1$ and lower, the bracket involving an element $f\otimes X \in \Gamma(\pi^\ast(\mf {gl}(E)))$ and an element $g\otimes \alpha \otimes e \in  \Gamma(\pi^\ast(\wedge^k E^\ast \otimes E))$ for $f,g \in C^\infty(E)$, $X \in \mf X_{lin}(E)$, $\alpha \in \Gamma(\wedge^k E^\ast)$, $e \in \Gamma(E)$ is given by 
$$
[f\otimes X,g\otimes \alpha \otimes e] = fX(g) \otimes \alpha \otimes e + fg \otimes X\cdot (\alpha \otimes e).
$$
The bracket involving two elements of $\Gamma(\pi^\ast(\mf {gl}(E)))$ is given by the action Lie algebroid bracket.
\end{prop}
\subsection{Linear vector fields preserving a subbundle}\label{sec:subbundle}
Let $\pi:E\to L$ be a vector bundle and $D\subset E$ a vector subbundle. In this section we combine sections \ref{sec:subsp} and \ref{sec:zerosec} to give a projective resolution of the subfoliation $\mc F_D\subset\mc F_L$ given by
$$
\mc F_D(E) = \{X\in \mc F_L(E) \mid X(I_D)\subset I_D\},
$$
where $I_D$ is the vanishing ideal of $D\subset E$. In other words, $\mc F_D(E)$ consists of all vector fields which are tangent to the subbundle $D$. Note that when $D = 0$, we are in the situation of section \ref{sec:zerosec}.\\
This can be approached in a similar way as $\mc F_L$: define $$\mf X_{lin}(E,D):= \{X\in \mf X_{lin}(E) \mid X(\Gamma(\text{Ann}(D))) \subset\Gamma(\text{Ann}(D))\},$$ where $\Gamma(\text{Ann}(D))\subset \Gamma(E^\ast)$ is viewed as a subset of $C^\infty(E)$.\\

Let for $i \geq 2$ $K_i\subset \wedge^i E^\ast \otimes E$ be the subbundle given by
$$
K_i:= \{\phi\in \wedge^i E^\ast \otimes E \mid \forall d\in D, \forall e_1,\dots, e_{i-1} \in E: \phi(d,e_1,\dots, e_{i-1})\in D\},
$$
Here the condition should be read fiberwise. Further, define 
$\mf {gl}(E,D) \subset \mf {gl}(E)
$ as the subbundle whose sections are precisely $\mf X_{lin}(E,D)$. Then the analogue of Proposition \ref{prop:subspres} holds, and we find:
\begin{prop}\label{prop:subbpres}
\begin{itemize}
    \item []
    \item[i)] For $j\geq 2$, the differential $$d_j:\Gamma(\pi^\ast(\wedge^j E^\ast \otimes E)) \to \Gamma(\pi^\ast(\wedge^{j-1}E^\ast \otimes E)) $$ as in \eqref{diag:resgle} restricts to a map
    $$
    d_j:\Gamma(\pi^\ast(K_j))\to \Gamma(\pi^\ast(K_{j-1})),
    $$
    and $d_2(\Gamma(\pi^\ast(K_2)))\subset \Gamma(\pi^\ast(\mf {gl}(E,D)))$.
    \item[ii)] The complex 
    \begin{equation}\label{eq:resgled}
        \begin{tikzcd}
        0 \arrow{r}& \Gamma(\pi^\ast(K_n)) \arrow{r}{d_n} &\Gamma(\pi^\ast(K_{n-1})) \arrow{r}{d_{n-1}}&\dots \arrow{r}{d_2}& \Gamma(\pi^\ast(\mf {gl}(E,D)))\arrow{r}{\rho_D} & \mc F_D(E) \arrow{r}& 0
        \end{tikzcd}
    \end{equation}
    is exact.
    \item[iii)] The bracket as described in Proposition \ref{prop:glebracket} restricts to \eqref{eq:resgled}.
\end{itemize}
Consequently, \eqref{eq:resgled} with the restrictions of the differential and bracket is a universal $L_\infty$-algebroid of the foliation $\mc F_D$, which is minimal at points of $L$.
\end{prop}
\subsection{Vector fields preserving a volume form}\label{sec:vol1}
In section \ref{sec:vol0} we constructed the universal $L_\infty$-algebroid for the foliation given by the action of $\mf {sl}(V)$ for an $n$-dimensional vector space $V$. Although we made the choice of a volume form, this was not strictly necessary in this case. 

Now if we want to generalize this example to the case of higher dimensional leaves, i.e. a linear foliation on a vector bundle $\pi:E\to L$ of rank $n$, such that the zero section is a leaf and the transverse foliation on the fibers $E_x$ for $x\in L$ is isomorphic to the one given by the action of $\mf{sl}(E_x)$, one approach would be to take an appropriate Lie subalgebroid $\mf{sl}(E,\mu)$ of $\mf {gl}(E)$ and then looking at the induced foliation on $E$. To generalize the special linear subalgebra, the sections of the Lie subalgebroid $\mf{sl}(E,\mu)$ should sit in a short exact sequence
\begin{equation}\label{diag:sessln}
\begin{tikzcd}
0\arrow{r} &\Gamma(End_0(E)) \arrow{r}{a}& \Gamma(\mf{sl}(E,\mu)) \arrow{r}{\rho}& \mf X(L)\arrow{r}& 0,
\end{tikzcd}
\end{equation}
where $End_0(E)$ is the kernel of the vector bundle map $\Tr:End(E) \to \mb R$ given by the trace. 
\subsubsection{The projective resolution}
Assume that $E$ is orientable and pick a volume form $\mu \in \Gamma(\wedge^n E^\ast)$. Then we can identify the Lie algebroid $\mf{sl}(E,\mu)$ as follows:
Let
$$
\mf X_{lin}^\mu(E) = \{X \in \mf X_{lin}(E) \mid X\cdot \mu = 0\},
$$
where $X\cdot \mu$ is defined as in Lemma \ref{lem:cdoact}. 
A local computation shows that there exists a vector bundle $\mf{sl}(E,\mu)$ satisfying $\eqref{diag:sessln}$ such that $$\Gamma(\mf{sl}(E,\mu)) = \mf X_{lin}^\mu(E)$$
analogous to Lemma \ref{lem:linvf}.
To construct the projective resolution, we adopt a similar approach as in the case where $L$ was a point. Define
$$
\widehat{\Tr}:\mf X_{lin}(E)\to \Gamma(\wedge^n E^\ast)
$$
by 
$$
\widehat{\Tr}(X) = - X\cdot \mu.
$$
Clearly, $\mf X_{lin}^\mu(E) = \ker(\widehat{\Tr})$. Moreover, this really extends the trace:
\begin{lem} \label{lem:trhatistr}For $A \in \Gamma(End(E))$,
$$
\widehat{\Tr}\circ a (A)= \Tr(A) \mu.
$$
\end{lem}
\begin{proof}
Pick a local frame $\{e_i\}_{i = 1}^n$ for $E$ and a dual frame $\{e^i\}_{i = 1}^n$ for $E^\ast$, such that $\mu = e^1\wedge \dots \wedge e^n$. Then 
\begin{align*}
    \widehat{\Tr}(a(A)) &= - \sum_{i = 1}^n e^1\wedge \dots \wedge\left.\frac{d}{dt}\right|_{t = 0}  e^i\circ exp(-tA) \wedge \dots \wedge e^n\\
    &= -\sum_{i = 1}^n e^1 \wedge \dots \wedge e^i\circ \left.\frac{d}{dt}\right|_{t = 0} exp(-tA) \wedge \dots \wedge e^n\\
    &= \sum_{i = 1}^n e^1 \wedge \dots \wedge e^i \circ A \wedge \dots \wedge e^n\\
    &= \Tr(A)\mu.
\end{align*}
\end{proof}
We will now apply the same ideas as in the case where $L$ is a point. As in Lemma \ref{lem:kerver}, the kernel of $\rho_\mu:\Gamma(\pi^\ast(\mf {sl}(E,\mu))) \to \mf X(E)$ is contained in $\Gamma(\pi^\ast(\End_0(E)))$, so we proceed in a similar way as in section \ref{sec:vol0}. Define for $i = 2,\dots, n$ the vector bundle map
$$
\phi_i: \wedge^i E^\ast\otimes E \to  \wedge^{i-1}E^\ast
$$
over $L$ by
$$
\phi_i(\alpha\otimes e) = (-1)^{i-1}\iota_e(\alpha)
$$
for $\alpha \in \wedge^i E^\ast, e \in E$. Setting $K_i = \ker (\phi_i)$, we obtain the following analog of Proposition \ref{prop:ressl}:
\begin{prop}\label{prop:sleres}
\begin{equation}\label{diag:ressle}
\begin{tikzcd}
0 \arrow{r} & \Gamma(\pi^\ast(\wedge^n E^\ast)) \arrow{r}{d_n\phi_n^{-1}\partial_n} & \Gamma(\pi^\ast(K_{n-1})) \arrow{r}{d_{n-1}} &\dots \arrow{r}{d_2} & \Gamma(\pi^\ast(\mf{sl}(E,\mu))) \arrow{r}{\rho_\mu} & \mc F_L^\mu(E) \arrow{r}& 0
\end{tikzcd}
\end{equation}
is exact.
\end{prop}
\subsubsection{The $L_\infty$-algebroid structure}
Analogous to section \ref{sec:bracketssl}, to define the $L_\infty$-algebroid structure on the resolution \eqref{diag:ressle}, we would like to restrict the bracket as described in Proposition \ref{prop:glebracket} to the kernel of the morphisms $\phi_k$ for $k\in \{1,\dots, n-1\}$. For elements of degrees $-1$ and lower and for two elements from $\Gamma(\pi^\ast(\mf {sl}(E,\mu)))$ this is clear, as this is just the fiberwise extension of Lemma \ref{lem:bracketrestr}. \\
For the bracket between $\Gamma(\pi^\ast(\mf{sl}(E,\mu)))$ and $\Gamma(\pi^\ast(K_q))$, take an element $f\otimes \alpha \otimes e \in \Gamma(K_q)$, where $f\in C^\infty(E), \alpha \in \Gamma(\wedge^q E^\ast), e \in \Gamma(E)$, and an element $ X\in \Gamma(\pi^\ast(\mf {sl}(E,\mu)))$, we compute
\begin{align*}
(-1)^{q-1}\phi_q([X,f\otimes \alpha \otimes e]) &= \rho_\mu(X)(f)\otimes \iota_e(\alpha) + f \otimes \iota_e(X\cdot\alpha) + f\otimes \iota_{X\cdot e}(\alpha) \\
&= \rho_\mu(X)(f)\iota_e(\alpha) + f \otimes X\cdot (\iota_e(\alpha))\\
&= 0,
\end{align*}
as $f\otimes \alpha \otimes e\in \Gamma(K_q)$ means that $\iota_e(\alpha) =0$. 
Hence, the bracket restricts to the subspaces given by the kernels of the $\phi_k$. Finally, analogous to section \ref{sec:bracketssl}, we use the natural action of $\Gamma(\pi^\ast(\mf {sl}(E,\mu)))$ on $\Gamma(\pi^\ast(\wedge^n E^\ast))\cong C^\infty(E)\mu$ to define the bracket between degree $0$ and $-n+1$.\\ Therefore, we again obtain a dg-Lie algebroid structure:
\begin{prop}\label{prop:slebracket}
The resolution \eqref{diag:ressle} of $\mc F_L^\mu(E)$ carries a dg-Lie algebroid structure, where the binary bracket is the restriction of the one described in Proposition \ref{prop:glebracket} when both entries have degrees $0,\dots, -n+2$, and the bracket $$[X,\tau]$$ of $X\in \Gamma(\pi^\ast\mf{sl}(E,\mu))$ with $\tau \in \Gamma(\pi^\ast(\wedge^n E^\ast))$ is given by the natural action of $X$ on $\tau$ as in Lemma \ref{lem:cdoact}.
\end{prop}
\subsection{Linear vector fields preserving a fiberwise symplectic form} \label{sec:sympleaf}
We now turn to the example of the symplectic case: let $\pi:E\to L$ be a symplectic vector bundle with $\omega \in \Gamma(\wedge^2 E^\ast)$ a non-degenerate skew-symmetric bilinear form. By now we know how to construct a Lie subalgebroid of $\mf {gl}(E)$ of linear vector fields preserving $\omega$: consider
$$
\mf X^\omega_{lin}(E) := \{X\in \mf X_{lin}(E)\mid X\cdot \omega = 0\},
$$
where $X\cdot \omega$ is defined as in Lemma \ref{lem:cdoact}. Note that $\mf X^\omega_{lin}(E)$ is closed under the Lie bracket of $\mf X_{lin}(E)$.\\
As in the previous section, there exists a vector bundle $\mf{sp}(E,\omega)$ over $L$, such that $$\Gamma(\mf {sp}(E,\omega)) = \mf X_{lin}^\omega(E).$$ We therefore obtain a Lie subalgebroid $\mf{sp}(E,\omega) \subset \mf {gl}(E)$ over $L$, which generates a linear foliation $\mc F_L^\omega$ on $E$. The zero section is a leaf, and the transverse foliation on $E_x$ for $x\in L$ is given by the standard $\mf {sp}(E_x,\omega_x)$-action on $E_x$.
\subsubsection{The projective resolution}
We proceed as in section \ref{sec:linsymp}. Define for $p = 1,\dots, n+1$ the vector bundle map $$\phi_p^\omega: \wedge^{p} E^\ast \otimes E \to  \wedge^{p-1} E^\ast \otimes  \wedge^2 E^\ast$$
over $L$, given by
$$
\phi_p^\omega(\alpha \otimes e) =  (-1)^{p-1}\iota_{e_i}(\alpha) \otimes e^i \wedge \iota_{e}(\omega),
$$
where $\{e_i\}_{i = 1}^n$ and $\{e^i\}_{i=1}^n$ are dual local frames of $E$ and $E^\ast$ respectively. Note that $\phi_p^\omega$ is independent of the choice of basis.\\

Define the differentials 
$$
d_p:\Gamma(\pi^\ast(\wedge^p E^\ast \otimes E)) \to \Gamma(\pi^\ast(\wedge^{p-1} E^\ast \otimes E))
$$
as in equation \eqref{eq:contreuler}, 
and
$$
\partial_p:\Gamma(\pi^\ast(\wedge^p E^\ast \otimes \wedge^2E^\ast)) \to \Gamma(\pi^\ast(\wedge^{p-1} E^\ast \otimes \wedge^2E^\ast)),
$$
given by 
$$
\partial_p(\alpha\otimes \tau) = -\iota_{\epsilon}(\alpha)\otimes \tau
$$
for $\alpha \in \Gamma(\pi^\ast(\wedge^p E^\ast))$, $\tau \in \Gamma(\pi^\ast(\wedge^2 E^\ast))$.
Then $\phi^\omega$ is a cochain map of degree $-1$, and setting $C_i := \coker(\phi^\omega_i)$, with induced differentials $\overline{\partial_\bullet}:\Gamma(\pi^\ast(C_p)) \to \Gamma(\pi^\ast(C_{p-1}))$ we can describe a projective resolution as in section \ref{sec:linsymp}:
\begin{prop}\label{prop:resspe}
The sequence
\begin{equation}\label{eq:resspe}
    \begin{tikzcd}
            0 \arrow{r}& \Gamma(\pi^\ast (C_{n+1})) \arrow{r}{\overline{\partial_n}} & \dots \arrow{r}{\overline{\partial_3}}  & \Gamma(\pi^\ast(C_3)) \arrow{r}{d_2\phi_2^{-1} \partial_2} & \Gamma(\pi^\ast(\mf{sp}(E,\omega))) \arrow{r}{\rho_\omega} &\mc F_L^\omega(E) \arrow{r}& 0
    \end{tikzcd}
\end{equation}
is exact. 
\end{prop}
\subsubsection{The (partial) $L_\infty$-algebroid structure}
To obtain the binary brackets, we follow the same approach as in section \ref{sec:linsymp}. Let 
$$
r^\omega_p:\Gamma(\wedge^p E^\ast \otimes \wedge^2 E^\ast) \to \Gamma(\wedge^{p+1}E^\ast \otimes E)
$$
be defined by
$$
r_p^\omega(\alpha_p \otimes \tau) = \left(\frac{1}{p+1} \alpha_p \wedge \iota_{e_i}(\tau) - \frac{(-1)^p}{p(p+1)} \iota_{e_i}(\alpha_p)\wedge \tau\right)\otimes \omega^{-1}(e^i) \in  \Gamma(\pi^\ast(\wedge^{p+1}E^\ast \otimes E)),
$$
for $\alpha \in \Gamma(\wedge^p E^\ast), \tau \in \Gamma(\wedge^2E^\ast)$. Then:
\begin{prop}\label{prop:spebracket}
Using the notation from section \ref{sec:bracketssp}, the binary operation $\llbracket-,-\rrbracket$ on \eqref{eq:resspe} defined by 
\begin{itemize}
    \item [-] The standard $\pi^\ast(\mf {sp}(E,\omega))$-action on itself and  $\pi^\ast(C_i)$ for $i = 3,\dots, n$ when one of the entries lies in $\Gamma(\pi^\ast (\mf {sp}(E,\omega)))$,
    \item [-] $$\llbracket \omega_p,\omega_q\rrbracket = [r_{p-1}^\omega \partial_pP_p(\omega_p),P_q(\omega_q)] + [P_p\omega_p,r_{q-1}\partial_qP_q(\omega_q)],$$
    for $\omega_p\in \Gamma(\pi^\ast(\wedge^p E^\ast \otimes \wedge^2 E^\ast)), \omega_q \in \Gamma(\pi^\ast(\wedge^q E^\ast \otimes \wedge^2 E^\ast))$ where $$P_p: \Gamma(\pi^\ast(\wedge^pE^\ast \otimes \wedge^2 E^\ast))\to \Gamma(\pi^\ast(\ker(r_p^\omega)))$$ is the projection $\text{id}-\phi_{p+1}^\omega\circ r_p^\omega$.
\end{itemize}
equips \eqref{eq:resspe} with a differential graded almost Lie algebroid structure, as in \cite[Definition 3.68]{univlinfty}.
\end{prop}
The same remarks as in section \ref{sec:linsymp} can be made:
\begin{rmk}
\begin{itemize} \item[]
    \item[-] The operation defined in Proposition \ref{prop:spebracket} satisfies the Jacobi identity if at least one entry has degree 0, but not when all of the entries have degree $\leq -1$. The analogue of \eqref{eq:terbrack} defines a ternary operation which serves as a contracting homotopy for the Jacobiator.
    \item[-] When the rank of $E$ is at most $4$, the Jacobi identity is trivially satisfied.
    \item[-] When the rank of $E$ is equal to $6$, the full $L_\infty$-algebroid structure is determined by the differential, the binary bracket as in Proposition \ref{prop:spebracket} and the analogue of \eqref{eq:terbrack}.
\end{itemize}
\end{rmk}
\section{The isotropy $L_\infty$-algebra in a singular point}\label{sec:generalfol}
Let $\mc F$ be a foliation on the vector space $V$. Assume that the origin $p\in V$ be a leaf of $\mc F$. In \cite[Section 4.2]{univlinfty} the authors define an $L_\infty$-algebra with trivial differential associated to a leaf of a foliation. Given a minimal resolution
\begin{equation}\label{diag:resf}
\begin{tikzcd}
0\arrow{r}& \Gamma(E_n) \arrow{r}{\partial_n}& \dots \arrow{r}{\partial_1} & \Gamma(E_0) \arrow{r}{\rho}& \mc F\arrow{r}& 0
\end{tikzcd}
\end{equation}
of $\mc F$ at $p$, and an $L_\infty$-algebroid structure $\{\ell_k,\rho\}_{k\in \mb N}$ on $\Gamma(E_\bullet)$, it is defined by restricting the multibrackets $\ell_k$ to the fibers $(E_i)_p$, which is well-defined because $\rho_p = 0$. This $L_\infty$-algebra is an invariant of $\mc F$, extending the isotropy Lie algebra $\mc F/I_p\mc F$, which is canonically isomorphic to $(E_0)_p$ with the restriction of $\ell_2$ (\cite[Proposition 4.14]{univlinfty}). In particular, the binary bracket $\ell_2$ turns $(E_i)_p$ into a $(E_0)_p$-representation. In this section we show that the spaces $(E_i)_p$ can be recovered directly from $\mc F$, without needing to find a projective resolution of $\mc F$. Moreover, we show that if $\mc F$ is linear, the $(E_0)_p$-representations on the $(E_i)_p$ can be determined explicitly. Note that the $(E_0)_p$-representation on $(E_0)_p$ is just the adjoint representation. This construction builds on \cite[Remark 4.9]{univlinfty} which states the following:
\begin{lem}\label{lem:toriso}
$$
(E_i)_p \cong \text{Tor}^{C^\infty(V)}_i(\mc F,\mb R),
$$
where the $C^\infty(V)$-module structure on $\mb R$ is defined by evaluation in the origin.
\end{lem}
\begin{proof}
One way to construct $\text{Tor}^{C^\infty(V)}_i(\mc F,\mb R)$ is to take a projective resolution
\[
\begin{tikzcd}
0\arrow{r}& \Gamma(E_n) \arrow{r}{\partial_n}& \dots \arrow{r}{\partial_1} & \Gamma(E_0) \arrow{r}{\rho}& \mc F\arrow{r}& 0
\end{tikzcd}
\]
of $\mc F$, then take the tensor product with $\mb R$ over $C^\infty(V)$ to obtain
\[
\begin{tikzcd}
0\arrow{r}& \Gamma(E_n)\otimes_{C^\infty(V)} \mb R \arrow{r}{\partial_n \otimes \text{id}}& \dots \arrow{r}{\partial_1\otimes\text{id}} & \Gamma(E_0)\otimes_{C^\infty(V)}\mb R \arrow{r}&  0
\end{tikzcd} 
\]
and compute the cohomology. As $\Gamma(E_i)\otimes_{C^\infty(V)} \mb R \cong (E_i)_p$ and the differentials become trivial, the result follows.
\end{proof}
It is however a well-known fact (see \cite[Theorem 2.7.2]{weibelhomalg} for instance) that instead of first taking a projective resolution of $\mc F$ and then taking the tensor product with $\mb R$, we can equivalently first take a projective resolution of $\mb R$, and then take the tensor product with $\mc F$ and compute cohomology. A major advantage here is that we know an explicit resolution of $\mb R$: it is given by the complex \eqref{eq:koszexact2}. We therefore obtain:
\begin{prop}\label{prop:toriso} For $i = 0,\dots, n$, we have
\begin{equation}\label{eq:toriso}
(E_i)_p \cong H^i(\Gamma(\wedge^\bullet V^\ast)\otimes_{C^\infty(V)}\mc F(V),d_\bullet \otimes \text{id})
\end{equation}
where $\wedge^{-1}(V^\ast)$ is understood to be 0, and $d =\iota_{x^i\partial_{x^i}}$.
\end{prop}
Next, we consider the action $\sigma_i$ of the isotropy Lie algebra $\mc F/I_p\mc F \cong (E_0)_p$ on $(E_i)_p$ by the binary bracket $\ell_2$. It turns out that when $\mc F$ is a \emph{linear} foliation, we can define a canonical action on the right hand side of \eqref{eq:toriso} which under the isomorphism of Proposition \ref{prop:toriso} corresponds to $\sigma_i$.
\begin{prop}\label{prop:repisoliealg}
Let $\mc F$ be a linear foliation on $V$. Then:
\begin{itemize}
    \item[i)] The map $$ \text{lin}: \mc F\to \mc F$$ given by
    $$
    X\mapsto X^{(1)} \in \mc F
    $$
    descends to an injective Lie algebra homomorphism 
    $$
    \overline{\text{lin}}:\mc F/I_p\mc F \to \mc F.
    $$
    Here $X^{(1)}$ denotes the linear part of the vector field $X$.
    \item[ii)] Let $i = 0,\dots, n$. For $X \in \mc F(V)$, $\alpha \otimes Y \in \Gamma(\wedge^iV^\ast) \otimes_{C^\infty(V)} \mc F(V)$, the assignment
    \begin{equation}\label{eq:fnaction}
    (X+I_p\mc F)\cdot (\alpha \otimes \mc F) := [X^{(1)},\alpha \otimes Y]_{FN} = \mc L_{X^{(1)}}(\alpha) \otimes Y + \alpha \otimes [X^{(1)},Y]
    \end{equation}
    defines a representation of $\mc F/I_p\mc F$ on $\Gamma(\wedge^i V^\ast)\otimes_{C^\infty(V)} \mc F(V)$, compatible with the differential $d_\bullet$, where $[-,-]_{FN}$ is the Fr\"olicher-Nijenhuis bracket. Consequently, there is a well-defined action on the cohomology groups.
    \item[iii)] The $\mc F/I_p\mc F$-action on $H^i(\Gamma(\wedge^\bullet V^\ast)\otimes_{C^\infty(V)} \mc F(V),d_\bullet\otimes \text{id})$ induced by the $\mc F/I_p\mc F$-action \eqref{eq:fnaction} is equivalent to the $\mc F/I_p\mc F$-action on $(E_i)_p$.
\end{itemize}
\end{prop}
\begin{proof}
\begin{itemize}
\item[]
    \item [i)] The fact that $\overline{lin}$ is a well-defined Lie algebra homomorphism was shown in \cite[Section 4]{holtraf}. For the injectivity, we need to show that if a vector field $Y \in \mc F(V)$ vanishes quadratically, it can be written as a linear combination 
    $$
    Y  = \sum_{i = 1}^r f^iX_i,
    $$
    where $X_i\in \mc F(V)$, and $f^i(p) =0$ for $i = 1,\dots, r$. As $\mc F$ is a linear foliation, we can take the $X_i$ to be linear vector fields which are linearly independent over $\mb R$. Then
    $$
    0 = Y^{(1)} = \sum_{i =1}^r f^i(0)X_i, 
    $$
    which implies that $f^i(0) = 0$.
    \item[ii)] This follows directly from the fact that $\overline{lin}$ is a Lie algebra homomorphism, and that the Fr\"olicher-Nijenhuis bracket satisfies the Jacobi identity. The compatibility with the differentials follows from the fact that 
    $$
    [\mc L_{X^{(1)}},\iota_{x^i\partial_{x^i}}] = \iota_{[X^{(1)},x^i\partial_{x^i}]} = 0,
    $$
    as $X^{(1)}$ is linear.
    \item[iii)] For this, we recall the isomorphism \eqref{eq:toriso} as described in \cite{weibelhomalg}. Given the projective resolutions \eqref{eq:koszexact2} and \eqref{diag:resf}, we can take the tensor product to obtain a double complex $$(\Gamma(\wedge^\bullet V^\ast) \otimes_{C^\infty(V)} \Gamma(E_{\bullet}),d_\bullet \otimes \text{id}, \text{id} \otimes \partial_\bullet).$$ From the double complex, we can construct the total complex
    $$
    (\text{Tot}(\Gamma(\wedge ^\bullet V^\ast) \otimes_{C^\infty(V)} \Gamma(E_\bullet), d_\bullet\otimes \text{id} + \text{id}\otimes \partial_\bullet).
    $$
    Then the maps 
    $$
    \text{id} \otimes \rho: \text{Tot}(\Gamma(\wedge^\bullet V^\ast)\otimes_{C^\infty(V)} \Gamma(E_\bullet)) \to \Gamma(\wedge^\bullet V^\ast) \otimes_{C^\infty(V)} \mc F(V)
    $$
    and 
    $$
    \text{ev}_p\otimes \text{id}: \text{Tot}(\Gamma(\wedge^\bullet V^\ast)\otimes_{C^\infty(V)} \Gamma(E_\bullet)) \to \mb R \otimes_{C^\infty(V)} \Gamma(E_\bullet)
    $$
    induce isomorphisms in cohomology. As both maps are compatible with the $\mc F/I_p\mc F$-actions, the isomorphisms in cohomology respect the $\mc F/I_p\mc F$-action as well.
\end{itemize}
\end{proof}
The proposition now allows us to compute invariants of the foliation $\mc F(V)$ without needing an explicit resolution of $\mc F$, as we do in the following example.
\begin{exmp}
Consider on $V = \mb R^2$ the foliation $\mc F_1(V) = \langle x\partial_x, y\partial_x\rangle_{C^\infty(V)}$. Then by Lemma \ref{lem:toriso},
$$
(E_0)_p := \text{Tor}_0^{C^\infty(V)}(\mc F(V),\mb R) = \mc F_1/I_p\mc F_1.
$$
For $\text{Tor}_1^{C^\infty(V)}(\mc F(V),\mb R)$, a straightforward computation shows that the  middle cohomology of
\begin{equation}\label{eq:tor1}
\begin{tikzcd}
\Gamma(\wedge ^2 V^\ast) \otimes_{C^\infty(V)} \mc F(V) \arrow{r}{d_2} & \Gamma(V^\ast)\otimes_{C^\infty(V)} \mc F(V) \arrow{r}{d_1} &\mc F(V)
\end{tikzcd}
\end{equation}
is one-dimensional, generated by the class of $\gamma:=dx \otimes y\partial_x - dy \otimes x\partial_x$. Observe that this element is not exact in \eqref{eq:tor1}: although it can be written as 
$$
dx \otimes y\partial_x - dy \otimes x\partial_x = d_2(dx\wedge dy \otimes \partial_x)
$$
in \eqref{eq:resgln}, $\partial_x \not\in \mc F(V)$. Moreover, any exact element in \eqref{eq:tor1} must vanish at least quadratically in the origin, which is not the case for $\gamma$.\\
Finally, it is easy to see that $d_2$ is injective, so we now know that for any minimal resolution \eqref{diag:resf},
the space $(E_0)_p$ is two-dimensional, the space $(E_1)_p$ is one-dimensional, and the spaces $(E_i)_p$ for $i\geq 2$ are trivial. The Lie algebra structure on $(E_0)_p$ is the non-abelian two-dimensional Lie algebra, while the action of $(E_0)_p$ on $(E_1)_p$ is trivial.
\end{exmp}
\begin{exmp}
We can modify the previous example to obtain a foliation which is not linear: consider $\mc F_2(V) = \langle (x+xy)\partial_x+ y^2\partial_y,y\partial_x\rangle_{C^\infty(V)}$. It is not difficult to see that $\mc F_2(V)$ is a projective $C^\infty(V)$-module. Consequently, for any minimal resolution \eqref{diag:resf}, $(E_0)_p$ is two-dimensional, and $(E_i)_p = 0$ for $i \geq 1$. Although it was already known that there exists no analytic diffeomorphism of $V$ taking the generators of $\mc F_2(V)$ to the generators of $\mc F_1(V)$ of the previous example (see \cite[Proposition 1.2]{Guillemin1968RemarksOA}), the above argument shows that there does not even exist a smooth diffeomorphism of $V$ taking the $C^\infty(V)$-\emph{module} $\mc F_2(V)$ to $\mc F_1(V)$, showing that not even the germs of the foliations $\mc F_1$ and $\mc F_2$ are equivalent, even though the modules generated by the first order approximations of the generators around $p\in V$ are equal. Of course, in this case the difference between $\mc F_1(V)$ and $\mc F_2(V)$ can be seen by considering the dimension of the regular leaves: for $\mc F_1$ they are 1-dimensional, while for $\mc F_2$ they are 2-dimensional.
\end{exmp}

\appendix
\section*{Appendix}
\section{Compatibility of $r^\omega$ with the differentials}\label{appendix1}
In this section we use the notation from section \ref{sec:linsymp}, and investigate whether the left inverse $r^\omega$ of $\phi^\omega$ can be chosen to be a cochain map in some degrees, which would simplify the brackets of the $L_\infty$-algebroid structure.\\

As the choice of $r^\omega$ in \eqref{eq:spbracket2} is not unique, we investigate whether the left inverse $r^\omega$ can be chosen to be compatible with the differentials, as this would force $\llbracket-,-\rrbracket$ to be equal to $\{-,-\}$.\\
However, it is clear that this is not possible in all degrees: first of all, as $r_1^\omega$ is not only a left inverse, but the unique inverse, as $\phi_2^\omega$ is an isomorphism. Hence, there is no choice there. Then, the existence of $\widetilde{r_2^\omega}:\Gamma(\wedge^2V^\ast \otimes \wedge^2 V^\ast) \to \Gamma(\wedge^3 V^\ast \otimes V)$ such that
$$
r_1^\omega\partial_2 + d_3\widetilde{r_2^\omega} = 0
$$
implies that $d_2r_1^\omega \partial_2 = 0$, which is not the case. \\ Nevertheless, we consider the other degrees, as compatibility with the differentials would simplify the binary and ternary brackets.\\
We start with the lowest degree: Let $n = \dim V$. In degrees $-n$ and $-n+1$, we get the following square
\[
\begin{tikzcd}
0\arrow{r}\arrow{d} & \Gamma(\wedge^n V^\ast \otimes V)\arrow{d}{\phi_{n}^\omega} \\
\Gamma(\wedge^n V^\ast \otimes \wedge^2V^\ast) \arrow{r}{\partial_n}& \Gamma(\wedge^{n-1}V^\ast \otimes \wedge^2 V^\ast)
\end{tikzcd}.
\]
Given a left inverse $\widetilde{r_{n-1}^\omega}:\Gamma(\wedge^{n-1} V^\ast \otimes \wedge^2 V^\ast) \to \Gamma(\wedge^n V^\ast \otimes V)$ of $\phi^\omega_n$ such that $$\widetilde{r_{n-1}^\omega}\partial_n = 0,$$ we note that the constant extension of the value at the origin $\widetilde{r_{n-1}^\omega}(0)$ is also a left inverse of $\phi_n^\omega$ which satisfies $$
\widetilde{r_{n-1}^\omega}(0)\partial_n = 0.
$$
It therefore suffices to show that there exists no constant (in $V$) left inverse $\widetilde{r_{n-1}^\omega}$ of $\phi_n^\omega$ such that 
$$
\widetilde{r_{n-1}^\omega}\partial_n = 0.
$$
Let $\mu\in \wedge^n V^\ast, \tau \in \wedge^2 V^\ast$. Then we can view $\partial_n$ as an injective $\mb R$-linear map
$$
\partial_n: \wedge^n V^\ast \otimes \wedge^2 V^\ast \to V^\ast \otimes \wedge^{n-1}V^\ast \otimes \wedge^2 V^\ast,
$$
as $\partial_n$ has linear coefficient functions, and $\widetilde{r_{n-1}^\omega}$ extends by $\text{id}_{V^\ast}\otimes \widetilde{r_{n-1}^\omega}$ to a map
$$
\text{id}_{V^\ast}\otimes \widetilde{r_{n-1}^\omega}:V^\ast \otimes \wedge^{n-1}V^\ast \otimes \wedge^2 V^\ast \to V^\ast \otimes \wedge^n V^\ast \otimes V
$$
by $C^\infty(V)$-linearity.\\
Now 
$$
\text{id}_{V^\ast}\otimes \widetilde{r_{n-1}^\omega}\partial_n(\mu\otimes \tau) = 0
$$
implies that 
$$
e^i \otimes \iota_{e_i}(\mu) \otimes \tau \in \ker(\text{id}_{V^\ast}\otimes \widetilde{r_{n-1}^\omega}).
$$
However, as $\ker(\text{id}_{V^\ast}\otimes \widetilde{r_{n-1}^\omega}) = V^\ast \otimes \ker(\widetilde{r_{n-1}^\omega})$, it follows that 
$$
\iota_{e_i}(\mu) \otimes \tau \in \ker( \widetilde{r_{n-1}^\omega})
$$
for each $i = 1,\dots, n$.\\
These elements actually generate the entirety of $\wedge^{n-1} V^\ast \otimes \wedge^2 V^\ast$, forcing $ \widetilde{r_{n-1}^\omega} = 0$, contradicting the assumption that $\widetilde{r_{n-1}^\omega}\phi_n^\omega = \text{id}$.\\
Now fix $\dim V = 4$. The general case discussed above show that there exist no left inverse $\widetilde{r_3^{\omega}}$ of $\phi_4^\omega$ such that 
$$
d_4\widetilde{r_3^\omega}\partial_4= 0.
$$
We will show that there exists no \emph{$\mf{sp}(V,\omega)$-equivariant} left inverse $\widetilde{r_2^\omega}$ of $\phi_3^\omega$ satisfying
$$
d_3 \widetilde{r_2^\omega}\partial_3 = 0.
$$
The requirement that $\widetilde{r_2^\omega}$ is $\mf{sp}(V,\omega)$-equivariant is natural, as $\phi^\omega_3$ is. We follow \cite[Chapter 16]{harris1991representation} to determine the space of all $
\mf {sp}(V,\omega)
$-equivariant maps $\wedge^2V^\ast \otimes \wedge^2 V^\ast \to \wedge^3 V^\ast \otimes V$, and then restrict to those which are left inverses of $\phi^\omega_3$. For this, we decompose the respective spaces into irreducible $\mf{sp}(V,\omega)$-representations:
\begin{lem}
\begin{align*}
    R_1:=\wedge^3 V^\ast \otimes V &\cong \mb R \oplus W \oplus S^2(V)\\
    R_2:=\wedge^2 V^\ast \otimes \wedge^2 V^\ast & \cong \mb R^{\oplus 2} \oplus W^{\oplus 2} \oplus S^2(V) \oplus C,
\end{align*}
where $W = \text{Ann}(\mb R\omega)\subset \wedge^2 V$, and $C$ is an irreducible representation not isomorphic to $\mb R$, $W$ or $S^2(V)$.
\end{lem}
Now we would like to apply a variation of Schur's lemma (see for instance \cite{Humphreys1973-ub}) to compute the space of $\mf {sp}(V,\omega)$-equivariant maps $\wedge^2 V^\ast \otimes \wedge^2 V^\ast \to \wedge^3 V^\ast \otimes V $. We first obtain:
\begin{lem}
\begin{align*}\Hom_{\mf{sp}(V,\omega)}(R_2,R_1) &\cong \End_{\mf {sp}(V,\omega)}(\mb R)^{\oplus 2} \oplus \End_{\mf {sp}(V,\omega)}(W)^{\oplus 2} \oplus \End_{\mf {sp}(V,\omega)}(S^2(V))\\
&\cong \mb R^2 \oplus \mb R^2 \oplus \mb R.\end{align*}
\end{lem}
\begin{proof}
By Schur's lemma, the restriction of a map of representation to irreducible factors is either 0, or an isomorphism, which proves the first isomorphism in the statement. For the second isomorphism, we observe that when complexifying, the representations 
$$
\mb C, W\otimes_{\mb R} \mb C, S^2_{\mb C}(V\otimes_{\mb R} \mb C)
$$
are irreducible $\mf{sp}(V\otimes_{\mb R} \mb C,\omega)$-representations, where $\omega$ is now extended to a $\mb C$-bilinear skew-symmetric map
$$
\omega: V\otimes_{\mb R} \mb C \times V\otimes_{\mb R} \mb C \to \mb C.
$$
Moreover, it is easy to see that for any representation $T$, the natural map
$$
\End_{\mf {sp}(V,\omega)}(T)\otimes_{\mb R} \mb C \to  \End_{\mf {sp}(V\otimes_{\mb R} \mb C,\omega)}(T\otimes_{\mb R} \mb C)
$$
is an isomorphism. As the endomorphism ring of a complex irreducible representation is $\mb C$ by Schur's Lemma, it follows that the endomorphism ring of the real representations $\mb R, W, S^2(V)$ is $\mb R$, concluding the proof of the lemma.
\end{proof}
We explicitly construct the generators of $\Hom_{\mf{sp}(V,\omega)}(R_2,R_1)$: pick a basis $\{e_i\}_{i=1}^4$ of $V$ such that $\omega = e^1\wedge e^3 + e^2 \wedge e^4$, and let 
$$
\pi_\omega = \frac{1}{2}(e_1 \otimes e_3 + e_2 \otimes e_4 - e_3\otimes e_1-e_4 \otimes e_2)\in V\otimes V.
$$
\begin{lem}
Let $\tau \in \wedge^2 V^\ast$. Define $$ \overline{\tau}:= \tau - \frac{1}{2}(\tau(e_{1},e_3)+\tau(e_2,e_4))\omega.$$\\
Let $\tau_1,\tau_2\in \wedge^2 V^\ast$. Then $\Hom_{\mf{sp}(V,\omega)}(R_2,R_1)$ is generated by the maps 
\begin{align*}
    p_1(\tau_1\otimes \tau_2) &= \frac{1}{4}(\tau_1(e_1,e_3)+\tau_1(e_2,e_4))(\tau_2(e_1,e_3)+\tau_2(e_2,e_4))\pi_\omega,\\
    p_2(\tau_1 \otimes \tau_2) &= (\overline{\tau_1}\wedge \overline{\tau_2})(e_1,e_3,e_2,e_4)\pi_\omega,\\
    q_1(\tau_1 \otimes \tau_2) &= ((\omega^{\flat})^{-1}\wedge (\omega^\flat)^{-1})(\overline{\tau_1})\frac{1}{2}(\tau_2(e_1,e_3) + \tau_2(e_2,e_4)),\\
    q_2(\tau_1\otimes \tau_2) & = \frac{1}{2}(\tau_1(e_1,e_3)+\tau_1(e_2,e_4))((\omega^{\flat})^{-1}\wedge (\omega^\flat)^{-1})(\overline{\tau_2}),\\
    s(\tau_1\otimes \tau_2) & = \overline{\tau_1}((\omega^\flat)^{-1}(e^k),e_j)\overline{\tau_2}(e_k,e_l)\omega^{-1}(e^j)\cdot \omega^{-1}(e^l),
\end{align*}
where $p_1,p_2$ correspond to the trivial representation, $q_1,q_2$ to $W$, and $s$ to $S^2(V)$. Here $\cdot $ denotes the symmetric product in $S^2(V)$, $\wedge^3 V^\ast \otimes V$ is identified with $V\otimes V$ via the volume form $\frac{1}{2}\omega \wedge \omega$, and $\wedge^2 V$ and $S^2(V)$ sit inside $V\otimes V$ as 
\begin{align*}
v_1\wedge v_2 &\mapsto \frac{1}{2}(v_1\otimes v_2 - v_2\otimes v_1),\\
v_1 \cdot v_2 & \mapsto \frac{1}{2}(v_1\otimes v_2 + v_2 \otimes v_1).
\end{align*}
\end{lem}
The lemma above allows us to formulate a condition under which 
\begin{equation}\label{eq:modmap}
\lambda_1p_1 +\lambda_2p_2 + \mu_1 p_1 + \mu_2 p_2 + \nu s 
\end{equation}
is a left inverse of $\phi_3^\omega$:
\begin{lem}
\eqref{eq:modmap} is a left inverse of $\phi^\omega_3$ if and only if 
\begin{align*}
    \lambda_1 &= 2-10\lambda_2,\\
    \mu_1 &= \mu_2-2,\\
    \nu &= -2.
\end{align*}
\end{lem}
It is now straightforward to show that there is no value of $\lambda_2,\mu_2 \in \mb R$ such that the corresponding map $\widetilde{r_2^\omega} = (2-10\lambda_2)p_1 + \lambda_2 p_2 + (\mu_2-2)q_1 + \mu_2q_2 -2 s$ satisfies
$$
d_3\widetilde{r_2^\omega}\partial_3 = 0.
$$
Consequently:
\begin{prop}\label{prop:nosplitdiff}
When $\dim V = 4$, there exist no $$\widetilde{r_2^\omega}:\Gamma(\wedge^2 V^\ast \otimes \wedge^2 V^\ast) \to \Gamma(\wedge^3 V^\ast \otimes V), \widetilde{r^\omega_3}:\Gamma(\wedge^3 V^\ast \otimes \wedge^2 V^\ast) \to\Gamma( \wedge^4 V^\ast \otimes V)$$ satisfying $$\widetilde{r_2^\omega}\partial_3 + d_4\widetilde{r_3^\omega} = 0.$$
\end{prop}
\bibliography{bib.bib}{}

\begin{thebibliography}{LGLS20}

\bibitem[AS09]{holgpd}
Iakovos Androulidakis and Georges Skandalis.
\newblock The holonomy groupoid of a singular foliation.
\newblock {\em J. Reine Angew. Math.}, 626:1--37, 2009.

\bibitem[AZ14]{holtraf}
Iakovos Androulidakis and Marco Zambon.
\newblock Holonomy transformations for singular foliations.
\newblock {\em Advances in Mathematics}, 256:348--397, 2014.

\bibitem[AZ16]{sheafcomp}
Iakovos Androulidakis and Marco Zambon.
\newblock {Stefan–Sussmann singular foliations, singular subalgebroids and
  their associated sheaves}.
\newblock {\em International Journal of Geometric Methods in Modern Physics},
  04 2016.

\bibitem[FH91]{harris1991representation}
W.~Fulton and W.F.J. Harris.
\newblock {\em Representation Theory: A First Course}.
\newblock Graduate Texts in Mathematics. Springer New York, 1991.

\bibitem[GS68]{Guillemin1968RemarksOA}
Victor~W. Guillemin and Shlomo Sternberg.
\newblock {Remarks on a paper of Hermann}.
\newblock {\em Transactions of the American Mathematical Society},
  130:110--116, 1968.

\bibitem[Hum73]{Humphreys1973-ub}
James~E. Humphreys.
\newblock {\em {Introduction to Lie algebras and representation theory}}.
\newblock Graduate Texts in Mathematics. Springer, New York, NY, January 1973.

\bibitem[KSM02]{cdoksmk}
Y.~Kosmann-Schwarzbach and K.~C.~H. Mackenzie.
\newblock Differential operators and actions of {Lie} algebroids.
\newblock In {\em Quantization, Poisson brackets and beyond. London
  Mathematical Society regional meeting and workshop on quantization,
  deformations, and new homological and categorical methods in mathematical
  physics, Manchester, UK, July 6--13, 2001}, pages 213--233. Providence, RI:
  American Mathematical Society (AMS), 2002.

\bibitem[Lav16]{lavau2016}
Sylvain Lavau.
\newblock {\em Lie infini-algébroides et feuilletages singuliers}.
\newblock PhD thesis, 2016.
\newblock Thèse de doctorat dirigée par Strobl, Thomas et Samtleben, Henning
  Mathématiques Lyon 2016.

\bibitem[Lav22]{modclasslavau}
Sylvain Lavau.
\newblock {The modular class of a singular foliation}.
\newblock {\em arXiv e-prints}, page arXiv:2203.10861, March 2022.

\bibitem[LGLS20]{univlinfty}
Camille Laurent-Gengoux, Sylvain Lavau, and Thomas Strobl.
\newblock {The universal Lie $\infty$-algebroid of a singular foliation}.
\newblock {\em Doc. Math.}, 25:1571--1652, 2020.

\bibitem[LS93]{shlie}
Tom Lada and Jim Stasheff.
\newblock Introduction to sh {Lie} algebras for physicists.
\newblock {\em Int. J. Theor. Phys.}, 32(7):1087--1103, 1993.

\bibitem[Sus73]{sussmann}
Hector~J. Sussmann.
\newblock Orbits of families of vector fields and integrability of systems with
  singularities.
\newblock {\em Bull. Am. Math. Soc.}, 79:197--199, 1973.

\bibitem[Vor10]{voronov}
Th. Voronov.
\newblock {$Q$-manifolds and higher analogs of {Lie} algebroids}.
\newblock In {\em XXIX workshop on geometric methods in physics,
  Bia{\l}owie\.za, Poland, June 27 -- July 3, 2010. Selected papers based on
  the presentations at the workshop}, pages 191--202. Melville, NY: American
  Institute of Physics (AIP), 2010.

\bibitem[{Wan}]{wang}
{Wang, Roy}.
\newblock {\em {On Integrable Systems \& Rigidity for PDEs with Symmetry}}.
\newblock PhD thesis.

\bibitem[Wei95]{weibelhomalg}
C.A. Weibel.
\newblock {\em {An Introduction to Homological Algebra}}.
\newblock Cambridge Studies in Advanced Mathematics. Cambridge University
  Press, 1995.

\end{thebibliography}
\bibliographystyle{alpha}
\Addresses

\end{document}